\documentclass[12pt,draft]{amsart}
\usepackage{amsmath,amsfonts,amssymb,amsthm,bbm,latexsym,mathrsfs}
\usepackage{graphicx,color,epsfig,fancyhdr,dsfont}
\usepackage{hyperref}
\hypersetup{
    colorlinks,
    linkcolor={red!50!black},
    citecolor={blue!50!black},
    urlcolor={blue!80!black}
}

\newtheorem{theorem}{Theorem}[section]
\newtheorem{lemma}[theorem]{Lemma}

\newtheorem{Proposition}[theorem]{Proposition}
\theoremstyle{definition}

\newtheorem{example}[theorem]{Example}

\newtheorem{remark}[theorem]{Remark}

\numberwithin{equation}{section}

\RequirePackage{amsmath}
\RequirePackage{amssymb}
\RequirePackage{amsthm}
\RequirePackage{epsfig}

\begin{document}

\date{}
\title[Dirac and Laplace type operators]
{The local counting function of operators of Dirac and Laplace type}

\author{Liangpan Li and Alexander Strohmaier}

\address{Department of Mathematical Sciences,
Loughborough University, LE11 3TU, UK}
 \email{l.li@lboro.ac.uk, a.strohmaier@lboro.ac.uk}



\date{}

\begin{abstract}
Let $P$ be a non-negative self-adjoint Laplace type operator
acting on sections of a hermitian vector bundle over a closed Riemannian manifold.
 In this paper we review the close relations between various $P$-related
  coefficients such as the  mollified spectral counting coefficients,
  the  heat trace coefficients, the resolvent trace coefficients,
  the residues
 of the spectral zeta function as well as certain Wodzicki residues.
 We then use the Wodzicki residue to obtain results about the local counting function of operators of Dirac and Laplace type.
   In particular,
we express the second term of the mollified spectral counting function of Dirac type operators
 in terms of geometric quantities and characterize those Dirac type operators for which this
 coefficient vanishes.
 \end{abstract}

\maketitle


\section{Introduction}

Let $M$ be a closed Riemannian manifold of dimension $d$ and with metric $g$.
Let $E$ be a smooth complex hermitian vector bundle over $M$.
As usual we denote by $C^\infty(M;E)$ the space of smooth sections of $E$, by $L^2(M;E)$
the Hilbert space of square integrable sections equipped with the natural inner product
defined by the hermitian structure on the fibres and the metric measure $\mu_g$ on $M$.

A second order partial differential operator  $P: C^\infty(M;E) \to C^\infty(M;E)$
is said to be of Laplace type if
its principal symbol $\sigma_P$ is of the form $\sigma_P(\xi) = g_x(\xi,\xi) \mathrm{id}_{E_x}$
for all covectors
$\xi \in T^*_x M$. In local coordinates this means that $P$ is of the form
\begin{equation}\label{Laplace local}
 P =-g^{ij}(x)\partial_i \partial_j + a^k(x)\partial_k + b(x),
\end{equation}
where $a^k$, $b$ are smooth matrix-valued functions,  and we have used Einstein's sum convention.
Given a Laplace type operator $P$, it is well known that there exist a unique connection $\nabla$ on $E$
and a unique bundle endomorphism $V \in C^\infty(M;\mathrm{End}(E))$
such that
$P = \nabla^* \nabla + V$.
Then $P$ is called a generalized Laplace operator if $V=0$. We will
assume here that $P$ is self-adjoint and non-negative. Thus
there exists an orthonormal basis $\{\phi_j\}_{j=1}^{\infty}$
for $L^2(M;E)$ consisting of smooth eigensections such that
$
 P \phi_j = \lambda_j^2 \phi_j,
$
where $\lambda_j$ are chosen non-negative and correspond to the eigenvalues of the operator $P^{1/2}$.

A first order partial differential operator  $D: C^\infty(M;E) \to C^\infty(M;E)$
is said to be of Dirac type if its square is of Laplace type.
This means that the principal symbol
$\sigma_D$ of $D$ satisfies the Clifford algebra relations
\begin{equation}
 \sigma_D(\xi) \sigma_D(\eta) +  \sigma_D(\eta) \sigma_D(\xi) = 2 g_x(\xi,\eta) \mathrm{id}_{E_x} \quad (\xi,\eta \in T^*_x M).
\end{equation}
Given a Dirac type operator $D$,
we denote by $\gamma$ the action of the Clifford algebra bundle on $E$ generated by the principal symbol of $D$, and suppose
$\nabla$ is a connection on $E$ that is compatible with the Clifford action $\gamma$ (see Section \ref{Section 3} for details).
We call the triple $(E,\gamma,\nabla)$ a Dirac bundle, and can express $D$ uniquely as
$D = \gamma \nabla + \psi$,
where  $\psi \in C^\infty(M;\mathrm{End}(E))$ is  called the potential of $D$ associated with the Dirac bundle $(E,\gamma,\nabla)$.
In particular, $D$ is called the  generalized Dirac operator associated
with the Dirac bundle $(E,\gamma,\nabla)$ if $\psi=0$. We also assume $D$ is self-adjoint, which means
 there exists a discrete spectral resolution $\{\phi_j,\mu_j\}_{j=1}^{\infty}$ of $D$, where
  $\{\phi_j\}_{j=1}^{\infty}$ is an orthonormal basis for $L^2(M;E)$, and
$
 D \phi_j = \mu_j \phi_j
$ for each $j$.
Obviously, $\phi_j$ will be eigensections of $P=D^2$ with eigenvalues $\mu_j^2$. Therefore,
 using the notation from before $\lambda_j = |\mu_j|$.

Given a  classical (polyhomogeneous) pseudodifferential operator $A$ of order
 $m\in\mathbb{R}$ acting on sections of $E$,
 the microlocalized spectral counting function $N_A(\mu)$ of $D$ is defined as
\begin{equation}
 N_A(\mu) = \left \{ \begin{matrix} \sum_{0 \leq \mu_j < \mu} \langle A \phi_j, \phi_j \rangle & \textrm{if }\mu>0,\\
 \sum_{\mu \leq \mu_j < 0} \langle A \phi_j, \phi_j \rangle & \textrm{if } \mu \leq 0.  \end{matrix} \right.
\end{equation}
Thus, $N_A(\mu)$ is a piecewise constant function on $\mathbb{R}$ such that
$$
 N'_A(\mu) = \sum_{j=1}^{\infty} \langle A \phi_j, \phi_j \rangle \delta_{\mu_j},
$$
where
 $\delta_{\mu_j}$ denotes the delta function on $\mathbb{R}$ centered at $\mu_j$.
In case $A$ is the operator of multiplication by a function $f(x) \in C^\infty(M)$
then
\[
 N_A(\mu) = \int_M f(x) N_x(\mu) d \mu_g(x),
\]
where
\begin{equation}
 N_x(\mu) = \left \{ \begin{matrix} \sum_{0 \leq \mu_j < \mu} \| \phi_j(x) \|_{E_x}^2  & \textrm{if }\mu>0,\\
 \sum_{\mu \leq \mu_j < 0}\| \phi_j(x) \|_{E_x}^2 & \textrm{if } \mu \leq 0  \end{matrix} \right.
\end{equation}
 is the so-called local counting function of $D$.

 Let $\chi \in \mathscr{S}(\mathbb{R})$ be a Schwartz function such that
 the Fourier transform $\mathscr{F}\chi$ of $\chi$ is $1$ near the origin
and  $\mathrm{supp}(\mathscr{F}\chi)\subset(-\delta,\delta)$, where $\delta$
is a positive constant smaller than half the radius
of injectivity of $M$. It is well known (see e.g.
\cite{Duistermaat Guillemin 1975,Ivrii,Ivrii 2,SafarovV,Sandoval,Zelditch} for various special cases) that
\begin{equation}\label{Mollified asmptotics Dirac new}
 (\chi * N_A' )(\mu) = \sum_{j=1}^{\infty} \langle A \phi_j, \phi_j \rangle \chi(\mu-\mu_j) \sim \sum_{k=0}^\infty \mathscr{A}_k(A,D) \mu^{d+m-k-1}\ \ \  (\mu\rightarrow \infty).
\end{equation}
This can be derived from studying the Fourier integral operator representation
of $A\frac{\mathrm{Sign}(D)+\mathrm{Id}_E}{2}e^{-\mathrm{i}t|D|}$ via the stationary phase method.
The mollified spectral counting coefficients $\mathscr{A}_k(A,D)$ do not depend on the choice of $\chi$, and are locally computable in terms of the total symbols of $A\frac{\mathrm{Sign}(D)+\mathrm{Id}_E}{2}$ and $|D|$. Note also the corresponding expansion for $\mu\rightarrow-\infty$ can be easily obtained from replacing $D$ with $-D$:
\begin{equation}
 (\chi * N_A' )(\mu)   \sim \sum_{k=0}^\infty \mathscr{A}_k(A,-D) |\mu|^{d+m-k-1}\ \ \  (\mu\rightarrow -\infty).
\end{equation}
Therefore, the function $\chi * N_A'$ contains all the information about $\{\mathscr{A}_k(A,\pm D)\}_{k=0}^{\infty}$.

One of the purposes of the paper is to show how to explicitly determine $\mathscr{A}_k(A,D)$. In particular,
for any bundle endomorphism $F$ of $E$, we can express $\mathscr{A}_1(F,D)$
 in terms of
geometric quantities such as $g,\gamma,\nabla,\psi,F$. To compare,
 Sandoval (\cite{Sandoval}) obtained an explicit expression of
 $\mathscr{A}_1(\mathrm{Id}_E,D)$, while Branson and Gilkey (\cite{Branson Gilkey}) can also do so for $\mathscr{A}_1(F,D)$ whenever
 $F$ is of the form $f\mathrm{Id}_E$ where $f$ is a smooth function on $M$. In the case of more general first order systems with
 $F = \mathrm{Id}_E$ a formula was also more recently obtained by Chervova, Downes and Vassiliev (\cite{Chervova JST}).
 In the case of operators of Dirac-type our paper also explains the relation between the results of \cite{Chervova JST} and the subsequent
 \cite{ChervovaDV} on one hand and known heat-trace invariants on the other.

The mollified spectral counting coefficients $\mathscr{A}_k(F,D)$ closely relate to the local counting function of $D$.
   Recall $D$ is a self-adjoint Dirac type operator with spectral resolution $\{\phi_j,\mu_j\}_{j=1}^{\infty}$.
 For each $j$ we denote by $\Psi_j=\Psi_j(x,y)$ the Schwartz kernel of the orthogonal projection onto the space
spanned by $\phi_j$, that is,
\[\langle \phi,\phi_j\rangle\phi_j(x)=\int_M\Psi_j(x,y)\phi(y)d\mu_g(y),\]
where $\phi\in C^{\infty}(M;E)$, $\Psi_j(x,y)$ maps $\phi(y)\in E_y$ to $\langle\phi(y),\phi_j(y)\rangle_{E_y}\phi_j(x)\in E_x$. We then denote by
 $\Phi_j=\Phi_j(x)=\Psi_j(x,x)\in C^\infty(M;\mathrm{End}(E))$ the diagonal restriction of $\Psi_j$.
Similar to  (\ref{Mollified asmptotics Dirac new}) there exists (see e.g. \cite{Ivrii}) an asymptotic expansion
\begin{equation}\label{16}
\sum_{j=1}^{\infty}\Phi_j\chi(\mu-\mu_j)\sim\ \sum_{k=0}^{\infty}\mu^{d-k-1}\mathscr{L}_k(D)\ \ \ (\mu\rightarrow\infty),
 \end{equation}
 where $\mathscr{L}_k(D)\in C^\infty(M;\mathrm{End}(E))$. This immediately implies that
\begin{equation}\label{17}
\mathscr{A}_k(F,D)=\int_M\mathrm{Tr}_E(F\mathscr{L}_k(D))\end{equation}
for all non-negative integers $k$. On the other hand,
 it is possible to recover $\mathscr{L}_k(D)$ from $\mathscr{A}_k(F,D)$ whenever explicit expressions of $\mathscr{A}_k(F,D)$  are known for all $F$.
 Note
   at each $x\in M$,
$\mathrm{Tr}_{E_x}(\Phi_j(x))=\|\phi_j(x)\|^2_{E_x}$. Thus
\begin{equation}
(\chi\ast N_x' )(\mu)\sim\sum_{k=0}^{\infty}\mu^{d-k-1}\mathrm{Tr}_{E_x}(\mathscr{L}_k(D)(x))\ \ \ (\mu\rightarrow\infty),
\end{equation}
whose asymptotic expansion coefficients $\mathrm{Tr}_{E_x}(\mathscr{L}_k(D)(x))$
are  determined by $\mathscr{L}_k(D)$.

The  mollified spectral counting  coefficients $\mathscr{A}_k(F,D)$  also closely relate to
 the small-time asymptotic expansion of the Schwartz kernels of $e^{-tD^2}$ and
 $De^{-tD^2}$.
Let $Q:C^{\infty}(M;E)\rightarrow C^{\infty}(M;E)$ be a partial differential operator of order $m$.
It is well known (see e.g. \cite{Fegan,Gilkey Invariance,Seeley 1966}) that  $Qe^{-t D^2}$ is a smoothing
operator with smooth kernel $K(t,x,y,Q,D^2)$ in the sense of
\begin{equation}\label{kernel}
(Qe^{-tD^2}\phi)(x)=\int_MK(t,x,y,Q,D^2)\phi(y)d\mu_g(y),
\end{equation}
where $K(t,x,y,Q,D^2)$ maps $E_y$ to $E_x$, and
the diagonal values of $K$ admit  a uniform small-time asymptotic expansion
\begin{equation}\label{heat kernel expansion}
K(t,x,x,Q,D^2)\sim\sum_{k=0}^{\infty}t^{\frac{k-d-m}{2}}{\mathscr H}_k(Q,D^2)(x)\ \ \ (t\rightarrow0^{+}),
\end{equation}
 where ${\mathscr H}_k(Q,D^2)\in C^\infty(M;\mathrm{End}(E))$ vanishes if $k+m$ is odd.
 It is not hard to show for all non-negative integers $k$ that
\begin{equation}\label{111}
\mathscr{L}_k(D)=\frac{\mathscr{H}_k(\mathrm{Id}_E,D^2)}{\Gamma(\frac{d-k}{2})}+\frac{\mathscr{H}_k(D,D^2)}{\Gamma(\frac{d+1-k}{2})},
\end{equation}
where $\Gamma$ denotes the Gamma function. This immediately implies that one is able to recover all the $\mathscr{H}_k(D,D^2)$
if the dimension $d$ of $M$ is odd and $\{\mathscr{H}_k(D,D^2)\}_{k=0}^d$ if $d$ is even from all the $\mathscr{L}_k(D)$.
To the authors' knowledge, although  how to express $\mathscr{H}_k(\mathrm{Id}_E,D^2)$, say for example $k\leq10$, is a well-studied topic (\cite{Gilkey AFSG,Gilkey 2007}),
it seems that no explicit formula for $\mathscr{H}_k(D,D^2)$ with $k$ odd is stated in the literature.

 Now we state the main results of the paper.  Let
$\sigma_A$, $\mathrm{Sub}(A)$ denote the principal and sub-principal symbol of $A$, respectively (See Section \ref{section 4} for details), and let $T_1^{\ast}M$ denote the unit cotangent bundle of $M$.
 For any self-adjoint Dirac type operator $D$ of potential $\psi$
 associated with the Dirac bundle $(E,\gamma,\nabla)$,  define
$
 \widehat \psi = \gamma(e_i) \psi  \gamma(e_i),
$
where $\{e_i\}_{i=1}^d$ is a local orthonormal frame in $T^*M$. Obviously, $\widehat \psi\in C^\infty(M;\mathrm{End}(E))$ is independent of
the choice of local orthonormal frames. Recall that $d$ denotes the dimension of $M$. If $k=0$, then it is easy to show that
\begin{equation}{\mathscr A}_0(A,D)=\frac{1}{2\cdot(2\pi)^{d}}\int_{T_1^{\ast}M}
\mathrm{Tr}_E\big(\sigma_A+\sigma_A\cdot\sigma_D\big).\end{equation}

\begin{theorem}\label{theorem 1.62} Let
 $D$ be a self-adjoint Dirac type operator and let $A$ be a classical pseudodifferential operator of order $m$ on sections of $E$. Then
\begin{align*}
{\mathscr A}_1(A,D)=&\frac{1}{2\cdot(2\pi)^{d}}\int_{T_1^{\ast}M} \mathrm{Tr}_E\Big(\mathrm{Sub}(A)-\frac{d+m-1}{2}\cdot\sigma_A\cdot\mathrm{Sub}(D^2)\Big)+\\
& \frac{1}{2\cdot(2\pi)^{d}}\int_{T_1^{\ast}M} \mathrm{Tr}_E\Big(\mathrm{Sub}(AD)-\frac{d+m}{2}\cdot\sigma_A\cdot\sigma_D\cdot\mathrm{Sub}(D^2)\Big).
\end{align*}
\end{theorem}

The subprincipal symbols appearing in the above expressions depend on a choice of local bundle frame. The integrals are however invariantly defined
as integration over the co-sphere at each point yields a well defined density on $M$ (see Section \ref{section 4}).
Note  that
\begin{equation}
\mathrm{Sub}(AB)=\mathrm{Sub}(A)\cdot\sigma_B+\sigma_A\cdot\mathrm{Sub}(B)+\frac{1}{2\mathrm{i}}\{\sigma_A,\sigma_B\},
\end{equation}
where $B$ is any other classical pseudodifferential operator on sections of $E$,  $\{\sigma_A,\sigma_B\}$
denotes the Possion bracket between $\sigma_A$ and $\sigma_B$. This means that ${\mathscr A}_1(A,D)$
depends only on the principal and sub-principal symbols of $A$ and $D$. In general,  our method of proof can also provide
an explicit formula for ${\mathscr A}_k(A,D)$ in terms of the local symbols of $A$ and $D$ for each $k\geq2$.
In the case when $A$ is a bundle endomorphism this specializes to the following theorem.

\begin{theorem}\label{theorem 1.6} Let
 $D$ be a self-adjoint Dirac type operator of potential $\psi$
 associated with the Dirac bundle $(E,\gamma,\nabla)$ and let $F$ be a smooth bundle endomorphism of $E$. Then
\begin{equation}{\mathscr A}_1(F,D)=\frac{1}{(4\pi)^{d/2}\cdot\Gamma(\frac{d}{2})}\int_M \mathrm{Tr}_E(F\cdot\frac{\widehat{\psi}-(d-2)\psi}{2}).\end{equation}
\end{theorem}

As immediate consequences of Theorem \ref{theorem 1.6}, one obtains
\begin{align}\label{L1D}
\mathscr{L}_1(D)=\frac{1}{(4\pi)^{d/2}\cdot\Gamma(\frac{d}{2})}\cdot\frac{\widehat{\psi}-(d-2)\psi}{2}
\end{align}
and
\begin{align}\label{114}
\mathscr{H}_1(D,D^2)=\frac{1}{(4\pi)^{d/2}}\cdot\frac{\widehat{\psi}-(d-2)\psi}{2}.
\end{align}




\begin{theorem}\label{theorem 1.3}  Let
 $D$ be a self-adjoint Dirac type operator.
 Then
${\mathscr L}_1(D)=0$ if and only if $D$ is a generalized Dirac operator.
\end{theorem}

Note that in the special case $\mathrm{rk}(E)=2$ and $d=3$ formula (\ref{L1D}) implies that pointwise ${\mathscr L}_1(D)$
is proportional to the identity matrix. Hence it vanishes if and only if the local coefficient
$\mathrm{tr} ({\mathscr L}_1(D))$ vanishes (see also Example \ref{example 55}).
Then Theorem \ref{theorem 1.3} directly implies the characterization theorem in
\cite{ChervovaDV} (see Example \ref{example 56} for details). Our result can therefore be seen as a generalization of Theorem 1.2 in
\cite{ChervovaDV} to higher dimensions and a more general geometric context.

Finally, we explain how to explicitly determine $\mathscr{A}_k(A,D)$.
Recall that $P$ is a self-adjoint
Laplace type operator with spectral resolution $\{\phi_j,\lambda_j^2\}_{j=1}^{\infty}$ and $A$ is
a classical pseudodifferential operator of order $m\in\mathbb{R}$.
Similar to (\ref{Mollified asmptotics Dirac new})
one has
\begin{equation}\label{Mollified asmptotics}
 (\chi * N_A' )(\lambda) = \sum_{j=1}^{\infty} \langle A \phi_j, \phi_j \rangle \chi(\lambda-\lambda_j) \sim \sum_{k=0}^\infty \mathscr{A}_k(A,P) \lambda^{d+m-k-1}\ \ \  (\lambda\rightarrow \infty),
\end{equation}
where the microlocalized spectral counting function $N_A(\lambda)$ of $P$ is  defined as
\begin{equation} N_A(\lambda) = \sum_{\lambda_j < \lambda} \langle A \phi_j, \phi_j \rangle.
\end{equation}
Then it is straightforward to show for all non-negative integers $k$ that
\begin{equation}\label{117}\mathscr{A}_k(A,D)=\mathscr{A}_k(A\frac{\mathrm{Sign}(D)+\mathrm{Id}_E}{2},D^2),\end{equation}
which means to determine $\mathscr{A}_k(A,D)$ it suffices to do so for $\mathscr{A}_k(A,D^2)$
for arbitrary classical pseudodifferential operators $A$.
Apart from the Fourier integral operator representation method introduced before, there exist several other ways to recover the mollified counting coefficients $\mathscr{A}_k(A,P)$.
First, the microlocalized spectral zeta function $\zeta(s,A,P)$ is defined by
\begin{equation}\label{zeta}
\zeta(s,A,P)=\sum_{\lambda_j>0}\frac{\langle A\phi_j,\phi_j\rangle}{\lambda_j^s}\ \ \ (\mathrm{Re}(s)>d+m).
\end{equation}
It is well-known   (see e.g. \cite{Duistermaat Guillemin 1975,Zelditch}) that $\zeta(s,A,P)$ admits a meromorphic continuation to $\mathbb{C}$ whose only singularities are simple poles at $s=d+m-k$ ($k=0,1,2,\ldots$) with residues
$\mathscr{A}_k(A,P)$. Second, the Mellin transform of
\[ \mathrm{tr}(A e^{-t P})-\sum_{\lambda_j=0}\langle A\phi_j,\phi_j\rangle\ \ \ (t\in(0,\infty))\]
 admits a meromorphic continuation $\zeta(2s,A,P)\Gamma(s)$ to $\mathbb{C}$ whose singularities
 can be completely determined from those of $\zeta(s,A,P)$ and $\Gamma(s)$.
 After establishing suitable vertical decay estimate
 for $\zeta(2s,A,P)\Gamma(s)$, one can deduce from the inverse Mellin transform theorem the following widely used heat  expansion  (see e.g. \cite{Grubb 1997,Grubb Seeley 1995,Grubb Seeley 1996,Loya,Scott})
\begin{align}\label{heat trace asymptotics}
 \mathrm{tr}(Ae^{-t P}) \sim \sum_{k=0}^\infty\big(\mathscr{B}_k(A,P) t^{\frac{k-d-m}{2}} + \mathscr{C}_k(A,P) t^k \log(t) + \mathscr{D}_k(A,P) t^k\big),
\end{align}
where $t\rightarrow0^+$, and the relations between
the mollified counting coefficients and some of the heat coefficients can be summarized as follows:\\
Case 1: If the order $m$ of $A$ is an integer, then
\begin{itemize}
\item $\mathscr{B}_k(A,P)=\frac{\Gamma(\frac{d+m-k}{2})}{2}\cdot\mathscr{A}_k(A,P)$\ \  ($d+m-k$ is positive or negative but odd);
\item $\mathscr{C}_k(A,P)=0$\ \ ($d+m+2k<0$);
\item $\mathscr{C}_k(A,P)=\frac{(-1)^{k+1}}{2\cdot k!}\cdot\mathscr{A}_{d+m+2k}(A,P)$\ \ ($d+m+2k\geq0$).
\end{itemize}
Case 2: If the order $m$ of $A$ is not an integer, then for all non-negative integers $k$:
\begin{itemize}
\item $\mathscr{B}_k(A,P)=\frac{\Gamma(\frac{d+m-k}{2})}{2}\cdot\mathscr{A}_k(A,P)$;
\item $\mathscr{C}_k(A,P)=0$.
\end{itemize}
Thus the heat expansion (\ref{heat trace asymptotics}) contains all the information about $\{\mathscr{A}_k(A,P)\}_{k=0}^{\infty}$.
In exactly the same way, the following  resolvent trace  expansion
(see e.g. \cite{Grubb Hansen 2002,Grubb Seeley 1995,Scott})
\[
 \mathrm{tr}(A(1+tP)^{-\frac{N}{2}}) \sim \sum_{k=0}^\infty\big(\mathscr{B}_k^{(N)}(A,P)t^{\frac{k-d-m}{2}}+\mathscr{C}_k^{(N)}(A,P)t^k \log(t)+\mathscr{D}_k^{(N)}(A,P)t^k\big)
\]
also contains all the information about $\{\mathscr{A}_k(A,P)\}_{k=0}^{\infty}$, where $N$ is any complex number such that  $\mathrm{Re}(N)>\max\{d+m,0\}$. To summarize, there exist at least four different ways
(mollified counting functions, spectral zeta functions, heat expansions, resolvent trace expansions) to
 determine  the mollified counting coefficients $\mathscr{A}_k(A,P)$.
For example, using parametrix constructions in any of these methods results in the well known leading term
\begin{equation}\label{leading term}
 \mathscr{A}_0(A,P)=\frac{1}{(2\pi)^d}\int_{T_1^{\ast}M}\mathrm{Tr}(\sigma_A).
\end{equation}

In this paper we will use the Wodzicki residue
 (see e.g \cite{Guillemin Wodzicki,KV,Schrohe Residue,Scott,Wodzicki,Wodzicki 84}) to study $\mathscr{A}_k(A,P).$
On the algebra $\Psi(M;E)$ of all classical pseudo-differential operators on $C^{\infty}(M;E)$,
there exists  a trace called Wodzicki's residue or non-commutative residue which is defined by
\begin{align}
\mathrm{res}(A)&=\int_M\mathrm{res}_x(A)dx,
\end{align}
where
\begin{align}
\mathrm{res}_x(A)dx&\triangleq\Big(\frac{1}{(2\pi)^d}\int_{|\xi|=1}\mathrm{Tr}(\sigma_{-d}(A)(x,\xi))dS(\xi)\Big)dx
\end{align}
is independent of the choice of local coordinates and thus is a global density on $M$,
$\sigma_{-d}(A)(x,\xi)$ denotes the homogeneous part of degree $-d$ of the total symbol of $A$,
$dS(\xi)$ denotes the sphere measure on $\mathbb{S}^{d-1}$.
If $M$ is connected any trace $\tau$ on $\Psi(M;E)$
is a multiple of $\mathrm{res}$.
It is well-known (see e.g. \cite{Ackermann,Gil,Grubb 1997,Kalau,KV,Schrohe Residue,Scott,Zelditch}) that
Wodzicki's residues are closely related to the mollified counting coefficients as well as the heat coefficients.
The connecting formula is
 (\cite{Guillemin Wodzicki,Wodzicki 84}, see also \cite[(0.2)]{Grubb 1997}, \cite[(1.2)]{Lesch}, \cite[(1.16)]{Schrohe Residue})
\begin{equation}\label{Res C0}
\mathrm{res}(A)=-2\mathscr{C}_0(A,P),
\end{equation}
which means $\mathscr{C}_0(A,P)$ is independent of the choice of $P$.
As for other coefficients, it is  known  for all non-negative integers $k$ that (see e.g. \cite[Thm. 5.2]{Gil})
\begin{equation}\label{thm 13}
\mathscr{C}_k(A,P)=\frac{(-1)^{k+1}}{2\cdot k!}\cdot\mathrm{res}(AP^k),
\end{equation}
and (see e.g. \cite[Prop. 4.2]{KV}, \cite[P. 106]{Scott})
\begin{equation}\label{thm 14}
\mathscr{A}_k(A,P)=\mathrm{res}(AP^{\frac{k-d-m}{2}}).
\end{equation}
Thus combining (\ref{117}) with  (\ref{thm 14}), one gets
\begin{equation}\label{Ak AD}
\mathscr{A}_k(A,D)=\mathrm{res}(A\frac{D+|D|}{2}|D|^{k-d-m-1}),
\end{equation}
which is the main tool for us to establish Theorem \ref{theorem 1.62}. We should mention that any of
the previous four methods can also be used to derive Theorem \ref{theorem 1.62}, but it seems that
these methods do not provide such a clear interpretation of $\mathscr{A}_k(A,P)$ for $k\geq2$.

This paper is structured as follows.

Section \ref{section 2} first reviews the construction of the Fourier integral operator $A e^{-\mathrm{i} t P^{1/2}}$
and the stationary phase expansion of $\chi * N'_A$, then
provide as consequences complete proofs of the singularity structures of the spectral zeta functions,
the heat expansions as well as the resolvent trace expansions.
As applications,  (\ref{thm 13}) and (\ref{thm 14}) follow from suitably applying formula (\ref{Res C0}).
The local counting function of operators of Dirac and Laplace type is also studied
in the last part of this section.


Section \ref{Section 3} first reviews the concept of the Dirac bundle and the associated Dirac type operators, then gives a direct proof of Theorem \ref{theorem 1.6} by appealing to  the Bochner--Weitzenb\"{o}ck technique in Riemannian geometry.

Section \ref{section 4}
proves Theorem \ref{theorem 1.62} and Theorem \ref{theorem 1.6} by applying the Wodzicki residue method.
The connections between certain  Wodzicki residues and the sub-principal symbols of
classical pseudodifferential operators are also discussed.

Section \ref{section 5}   characterizes
 Dirac type operators with vanishing second expansion coefficient.
 As an application, the so-called massless Dirac operators are studied.

Throughout the paper we assume on smooth sections of $E$,
\begin{itemize}
\item $P$ is a self-adjoint Laplace type operator with spectral resolution $\{\phi_j,\lambda_j^2\}_{j=1}^{\infty}$,
\item $D$ is a self-adjoint Dirac type operator with spectral resolution $\{\phi_j,\mu_j\}_{j=1}^{\infty}$ unless otherwise stated in the last section,
\item $A$ is a classical (polyhomogeneous) pseudodifferential operator of order $m$,
\item  $Q$ is a partial differential operator of order $m$,
\item $F$ is a smooth endomorphism.
\end{itemize}

\section{Mollified counting coefficients I: qualitative theory}\label{section 2}

This section  first reviews the construction of the Fourier integral operator $A e^{-\mathrm{i} t P^{1/2}}$
and the stationary phase expansion  of $\chi * N'_A$, then proves many consequences.

\subsection{FIO method}
Formula (\ref{Mollified asmptotics}) essentially is Proposition 2.1 in \cite{Duistermaat Guillemin 1975}, Corollary 2.2 in \cite{Ivrii}, Theorem 2.2 in \cite{Sandoval} and  Proposition 1.1 in \cite{Zelditch},
 except they either consider scalar operators or assume $A$ is of order zero.
  Recall that $\chi \in \mathscr{S}(\mathbb{R})$ is chosen so that $\mathscr{F}\chi=1$  near the origin
and  $\mathrm{supp}(\mathscr{F}\chi)\subset (-\delta,\delta)$, where $\delta$ is smaller than half the radius
of injectivity of $M$.
 If $t$ is sufficiently small, say $|t|<\delta_1<\delta$, then locally the integral kernel $(Ae^{-\mathrm{i}tP^{1/2}})(t,x,y)$ of the operator
 $Ae^{-\mathrm{i}tP^{1/2}}$ is well known to have the form
 \[(Ae^{-\mathrm{i}tP^{1/2}})(t,x,y)=\frac{1}{(2\pi)^d}\int_{\mathbb{R}^d}a(t,x,y,\xi)e^{\mathrm{i}\theta(t,x,y,\xi)}d\xi,\]
 where $a$ is a classical (matrix-valued) symbol of order $m$. The  scalar-valued phase function
  $\theta(t,x,y,\xi)=\kappa(x,y,\xi)-t\sigma_{P^{1/2}}(y,\xi)$, where $\kappa(y,y,\xi)=0$, was   introduced by H\"{o}rmander
  (\cite{Hormander 1968}). It is also known that $\mathrm{tr}(Ae^{-\mathrm{i}tP^{1/2}})$ is smooth in $(-\delta,\delta)\backslash\{0\}$, so
 we introduce a cut-off function $\varrho \in \mathscr{S}(\mathbb{R})$ satisfying $\varrho(t)=1$ if $|t|<\frac{\delta_1}{2}$ and  $\mathrm{supp}(\varrho)\subset(-\delta_1,\delta_1)$.
Using integration by parts  one gets
 \begin{align*}
 (\chi * N_A' )(\lambda) =&\frac{1}{2\pi}\int_{\mathbb{R}}(\mathscr{F}\chi)(t)(\varrho(t)+1-\varrho(t))\mathrm{tr}(Ae^{-\mathrm{i}tP^{1/2}})e^{\mathrm{i}\lambda t}dt\\
 =&\frac{1}{(2\pi)^{d+1}}\int_{M}\int_{\mathbb{R}^d}\int_{\mathbb{R}}(\mathscr{F}\chi)(t)\varrho(t)\mathrm{Tr}
 (a(t,y,y,\xi))e^{-\mathrm{i}t\sigma_{P^{1/2}}(y,\xi)}e^{\mathrm{i}\lambda t}dyd\xi dt\\
 &+o(\lambda^{-\infty})\ \ \ (\lambda\rightarrow\infty),
 \end{align*}
where $o(\lambda^{-\infty})$ is short for $o(\lambda^{-h})$ for all positive integers $h$.
Now proceed as in \cite{Duistermaat Guillemin 1975} via the stationary phase method to get (\ref{Mollified asmptotics}).

\subsection{Finite heat expansions}\label{sub 2.2}
It is  easy to prove that
 \begin{equation}\label{formula 21}\langle T, \varphi_t\rangle=\langle \rho\ast T, \varphi_t\rangle+o(t^{\infty})\ \ \ (t\rightarrow0^+), \end{equation}
where $T$ is a tempered distribution on $\mathbb{R}$, $\varphi\in\mathscr{S}(\mathbb{R})$, $\varphi_t(\lambda)=\varphi(t\lambda)$, and $\rho\in\mathscr{S}(\mathbb{R})$ is chosen so that $\mathscr{F} \rho=1$ near the origin.
Actually, this formula appears in equivalent forms in \cite{Duistermaat Guillemin 1975,Guillemin,Zelditch}, so its proof is omitted.
As an application, one easily gets
 \begin{equation}\label{formula 22}\langle N_A', \lambda^h\varphi(t\lambda)\rangle=\langle \chi\ast N_A', \lambda^h\varphi(t\lambda)\rangle+o(t^{\infty})\ \ \ (t\rightarrow0^+), \end{equation}
where $h$ is an arbitrary non-negative integer.
Note the left hand side of (\ref{formula 22}) just is $\mathrm{tr}(AP^{\frac{h}{2}}\varphi(t\sqrt{P}))$.
Now we claim the right  hand side of (\ref{formula 22}) is of the following form
 \begin{equation}
 \sum_{k<d+m+h}\mathscr{A}_k(A,P)\cdot\int_{0}^{\infty}
 \lambda^{d+m+h-k-1}\varphi(\lambda)d\lambda \cdot t^{k-d-m-h}+o(t^{\lceil m\rceil-m-1}).
 \end{equation}
 To this end we first decompose
\begin{align*}
\langle \chi\ast N_A', \lambda^h\varphi(t\lambda)\rangle=&\int_{-\infty}^0(\chi\ast N_A')(\lambda)\cdot  \lambda^h\varphi(t\lambda)d\lambda\ +\\
&\sum_{k<d+m+h}\int_{0}^{\infty} \mathscr{A}_k(A,P) \lambda^{d+m-k-1}\cdot\lambda^h\varphi(t\lambda)d\lambda\ + \\
& \int_{0}^{\infty}\big((\chi\ast N_A')(\lambda)-\sum_{k<d+m+h}\mathscr{A}_k(A,P) \lambda^{d+m-k-1}\big)\cdot\lambda^h\varphi(t\lambda)d\lambda\\
\triangleq &\ \alpha_1(t)+\alpha_2(t)+\alpha_3(t).
\end{align*}
Since $\mathrm{supp}(N_A')\subset[0,\infty)$, it is easy to verify that
$(\chi\ast N_A')(\lambda)=o(|\lambda|^{-\infty})$ as $\lambda\rightarrow-\infty$. This implies $\alpha_1(t)=O(1)$ as $t\rightarrow0^+$.
Note also
\[\alpha_2(t)=\sum_{k<d+m+h} \mathscr{A}_k(A,P)\cdot\int_{0}^{\infty} \lambda^{d+m+h-k-1}\varphi(\lambda)d\lambda \cdot t^{k-d-m-h}.\]
For simplicity we introduce
\[f(\lambda)=\big((\chi\ast N_A')(\lambda)-\sum_{k<d+m+h}\mathscr{A}_k(A,P) \lambda^{d+m-k-1}\big)\cdot\lambda^h\ \ \ (\lambda>0)\]
 and note $\alpha_3(t)=\int_0^{\infty}f(\lambda)\varphi(t\lambda)d\lambda$. To prove the claim we have two cases to consider.\\
Case 1: Suppose $d+m+h>0$. Let $\widetilde{k}$ be the unique non-negative integer such that
$\widetilde{k}<d+m+h\leq\widetilde{k}+1$.
Let $\beta\triangleq d+m+h-\widetilde{k}-1\in(-1,0]$, which implies there exists a  constant $C_1$ such that for all $\lambda\in(0,1)$, $|f(\lambda)|\leq C_1\lambda^{\beta}$.
  According to (\ref{Mollified asmptotics}),
 there exists another  constant $C_2$ such that for all $\lambda\geq1$, $|f(\lambda)|\leq C_2\lambda^{\beta-1}\leq C_2\lambda^{-1}$.
So as $t\rightarrow0^+$,
 \begin{align*}
|\alpha_3(t)|&\leq C_1\int_0^1\lambda^{\beta}|\varphi(t\lambda)|d\lambda+C_2\int_1^{\infty}\lambda^{-1}|\varphi(t\lambda)|d\lambda\\
&=C_1\int_0^t\lambda^{\beta}|\varphi(\lambda)|d\lambda\cdot t^{-(\beta+1)}+C_2\int_t^{\infty}\lambda^{-1}|\varphi(\lambda)|d\lambda\\
&=o(t^{-(\beta+1)})+O(|\log(t)|)\\
&=o(t^{-(\beta+1)}).
 \end{align*}
Since  $-(\beta+1)=\widetilde{k}-d-m-h=\lceil m\rceil-m-1<0$, we are done in the first case.\\
Case 2: Suppose $d+m+h\leq0$.  Obviously, there is a constant $C_3$ such that for all $\lambda\in(0,1)$, $|f(\lambda)|\leq C_3$.
  According to (\ref{Mollified asmptotics}),
$(\chi\ast N_A')(\lambda)=O(\lambda^{d+m-1})$ as $\lambda\rightarrow\infty$. Thus there is a constant $C_4$ such that for all $\lambda\geq1$, $|f(\lambda)|\leq C_4\lambda^{-1}$. So as $t\rightarrow0^+$,
 \begin{align*}
|\alpha_3(t)|&\leq C_3\int_0^1|\varphi(t\lambda)|d\lambda+C_4\int_1^{\infty}\lambda^{-1}|\varphi(t\lambda)|d\lambda\\
&=O(1)+O(|\log(t)|)\\
&=O(|\log(t)|).
 \end{align*}
But note $\lceil m\rceil -m -1<0$, we are also done in the second case.

 To summarize, we have shown for all $\varphi\in\mathscr{S}(\mathbb{R})$ as $t\rightarrow0^+$ that,
 \begin{equation}\label{formula 23}
 \mathrm{tr}(AP^{\frac{h}{2}}\varphi(t\sqrt{P}))=
\sum_{k<d+m+h}\frac{\mathscr{A}_k(A,P)\int_{0}^{\infty} \lambda^{d+m+h-k-1}\varphi(\lambda)d\lambda}{\displaystyle t^{d+m+h-k}}+o(t^{\lceil m\rceil-m-1}).
 \end{equation}

\subsection{Mellin transforms}
The Mellin transform of a continuous function $f$  on $(0,\infty)$ is the function $(Mf)(s)$ of the complex variable $s$, given by
\[(Mf)(s)=\int_0^{\infty}f(t)t^{s-1}dt\]
whenever the integral is well-defined.
An open strip $\Pi(\beta_1,\beta_2)=\{s\in\mathbb{C}:\beta_1<\mathrm{Re}(s)<\beta_2\}$ is called a basic strip of $Mf$ if the integral is absolutely convergent in that strip.
Given  meromorphic functions $u,v$ defined over an open subset $\Pi$ of $\mathbb{C}$, we denote
$u\asymp v\ (s\in\Pi)$ to mean $u-v$ is analytic in $\Pi$. For example, assume $f$ is a continuous function  on $(0,\infty)$
satisfying $f(t)=o(t^{-\infty})$ as $t\rightarrow\infty$ and assume there exist real numbers $\omega_0<\omega_1<\cdots<\omega_N$ such that
\begin{equation}\label{formula 31}
f(t)=\sum_{k=0}^{N-1} a_kt^{\omega_k}+O(t^{\omega_{N}})\ \ \ (t\rightarrow0^+).
\end{equation}
Then one may easily check that $\Pi(-\omega_0,\infty)$ is a basic strip of $Mf$, and $Mf$ admits a meromorphic continuation to $\Pi(-\omega_N,\infty)$
such that
\begin{equation}\label{formula 32}
(Mf)(s)\asymp\sum_{k=0}^{N-1} \frac{a_k}{s+\omega_k}\ \ \ (s\in\Pi(-\omega_N,\infty)).
\end{equation}
Actually,  for any $s\in\Pi(\gamma_1,\gamma_2)$ with $-\omega_N<\gamma_1<\gamma_2<\infty$, one has
\[
(Mf)(s)=\sum_{k=0}^{N-1} \frac{a_k}{s+\omega_k}+\int_0^1\big(f(t)-\sum_{k=0}^{N-1} a_kt^{\omega_k}\big)t^{s-1}dt+\int_1^{\infty}f(t)t^{s-1}dt,
\]
which immediately implies (\ref{formula 32}) as well as an upper bound estimate:
\begin{equation}\label{formula upper bound}
(Mf)(s)=O(1)\ \ \ (s\in\Pi(\gamma_1,\gamma_2),\ |s|\rightarrow\infty).
\end{equation}

\begin{lemma}[\cite{Mellin}]\label{lemma 31}
Let $f$ be a continuous function over $(0,\infty)$.  Assume there exist real numbers  $\beta_1<\beta_2<\beta_3<\beta_4$ such that
\begin{itemize}
\item $\Pi(\beta_3,\beta_4)$ is a basic strip of $Mf$,
\item $Mf$ admits a meromorphic continuation to $\Pi(\beta_1,\beta_4)$ with  finite poles there,
\item $Mf\asymp\sum_{(\omega,j)}\frac{C_{\omega,j}}{(s+\omega)^{j+1}}\ \ (s\in \Pi(\beta_2,\beta_4))$,
\item $Mf$ is analytic on $\mathrm{Re}(s)=\beta_2$,
\item $(Mf)(s)=O(|s|^{-2})\ \ (s\in\Pi_{\beta_1,\beta_4}, |s|\rightarrow\infty)$.
\end{itemize}
Then
\[f(t)=\sum_{(\omega,j)}C_{\omega,j}\Big(\frac{(-1)^{j}}{j!}t^{\omega}(\log t)^{j}\Big)+O(t^{-\beta_2})\ \ \ (t\rightarrow0^+).\]
\end{lemma}

\subsection{Spectral zeta functions}\label{sub 2.4}

In this part we deduce the well-known singularity structures of the spectral zeta functions.
Let $h$ be a positive integer such that $d+m+h>1$, and  define
\begin{equation}\label{trace fh}
f_h(t)=\mathrm{tr}(AP^{\frac{h}{2}}e^{-tP})=\sum_{\lambda_j>0} \langle A\phi_j,\phi_j\rangle \lambda_j^he^{-t\lambda_j^2}\ \ \ (t>0).
\end{equation}
Now we have two ways to study the Mellin transform of $f_h$. First,
we list $\{\lambda_j\}_{j=1}^{\infty}$ in non-decreasing order and thus
by Weyl's law (see e.g. \cite{BGV}),
$
\lim_{j\rightarrow\infty}{\lambda_j^d}/{j}
$
exists and is  positive. It is easy to see (see e.g. \cite{Calderon,Seeley 1966}) that there exists a  constant $C=C(A,P)$ such that
$|\langle A\phi_j,\phi_j\rangle|\leq C\lambda_j^m$ whenever $\lambda_j>0$.
Consequently,
 \begin{itemize}
\item  $f_h$ is  smooth  over $(0,\infty)$ with $f_h(t)=o(t^{-\infty})$ as $t\rightarrow\infty$;
\item the spectral zeta function $\zeta(s,A,P)$ is analytic in $\Pi(d+m,\infty)$;
\item $Mf_h$ has
a basic strip  $\Pi(\frac{d+m+h}{2},\infty)$ in which
$
(Mf_h)(s)=\zeta(2s-h,A,P)\Gamma(s).
$
\end{itemize}
Second, according to (\ref{formula 23})
$f_h$ is easily seen to have the form (\ref{formula 31}):
\begin{equation}
f_h(t)=\sum_{k=0}^{N-1}\frac{\Gamma(\frac{d+m+h-k}{2})}{2}\cdot \mathscr{A}_k(A,P)\cdot t^{\frac{k-d-m-h}{2}}+O(t^{\omega_{N}})\ \ \ (t\rightarrow0^+),
\end{equation}
where $N=N_h=d+h+\lceil m\rceil-1\geq1$, $\omega_N=\frac{\lceil m\rceil-m-1}{2}<0$, $\omega_k=\omega_N-\frac{N-k}{2}=\frac{k-d-m-h}{2}$ $(k=0,1,\ldots,N-1)$.
So by (\ref{formula 32}),  $Mf_h$ initially analytically defined in $\Pi(\frac{d+m+h}{2},\infty)$,
 admits a meromorphic continuation to $\Pi(-\omega_N,\infty)$ in which
\begin{equation}\label{formula 36}
(Mf_h)(s)\asymp\sum_{k=0}^{N-1} \frac{\Gamma(\frac{d+m+h-k}{2})}{2}\cdot \frac{\mathscr{A}_k(A,P)}{s-\frac{d+m+h-k}{2}}.
\end{equation}
Considering  in $\Pi(-\omega_N,\infty)$ the Gamma function $\Gamma$ is
analytic and has no zeros,
 it is easy to  see that $\zeta(s,A,P)$ initially analytically defined in $\Pi(d+m,\infty)$,
 admits a meromorphic continuation to $\Pi(-2\omega_N-h,\infty)$ in which
\begin{equation}\label{formula 37}
\zeta(s,A,P)=\frac{(Mf_h)(\frac{s+h}{2})}{\Gamma(\frac{s+h}{2})}\asymp\sum_{k=0}^{N_h-1} \frac{\mathscr{A}_k(A,P)}{s-(d+m-k)}.
\end{equation}
Letting $h\rightarrow\infty$ we get that $\zeta(s,A,P)$ admits a meromorphic continuation to $\mathbb{C}$ whose only singularities are simple poles at $s=d+m-k$ ($k=0,1,2,\ldots$) with residues
$\mathscr{A}_k(A,P)$.

\begin{remark}
Note that if $B$ is a smoothing operator on sections of $E$, then $\zeta(s,B,P)$ is an entire
function on $\mathbb{C}$, and consequently, $\mathscr{A}_k(A,P)=\mathscr{A}_k(A+B,P)$ for all $k$.
\end{remark}

For any $q\in\mathbb{R}$, we let $P^q:C^{\infty}(M;E)\rightarrow C^{\infty}(M;E)$
be the operator defined by functional calculus of $P$ if $q$ is non-negative and by sending $\phi$ to $\sum_{\lambda_j>0}\lambda_j^{2q}\langle \phi,\phi_j\rangle\phi_j$ if $q$ is negative. By a classical result by Seeley (\cite{Seeley 1966}) the operator
$P^q$ is a classical pseudodifferential operator of order $2q$.

 \begin{Proposition}\label{Prop 22}
 For any real number $q$, ${\mathscr A}_k(A,P)={\mathscr A}_k(AP^{q},P)$ holds for all non-negative integers $k$.
 \end{Proposition}

\begin{proof} Note  $AP^q$ is a classical pseudodifferential operator of order $m+2q$. Thus
$\zeta(s,AP^q,P)$ admits a meromorphic continuation to $\mathbb{C}$ whose only singularities are simple poles at $s=d+m+2q-k$ ($k=0,1,2,\ldots$) with residues
$\mathscr{A}_k(AP^q,P)$. But noting that $\zeta(s,AP^{q},P)=\zeta(s-2q,A,P)$, we must have  ${\mathscr A}_k(AP^{q},P)={\mathscr A}_k(A,P)$.
This finishes the proof of the proposition.
\end{proof}

\begin{remark} To determine ${\mathscr A}_k(A,P)$ we can assume without loss of generality  that  $A$ is of order zero
as ${\mathscr A}_k(AP^{-\frac{m}{2}},P)$,  which is ${\mathscr A}_k(A,P)$ by Proposition \ref{Prop 22}.
\end{remark}

\subsection{Full heat expansions}\label{heat expansion section}
In this part we provide another approach to the heat expansion (\ref{heat trace asymptotics}). Define a smooth function $f$ over $(0,\infty)$ by
\begin{equation}\label{trace f}
f(t)=\sum_{\lambda_j>0} \langle A\phi_j,\phi_j\rangle e^{-t\lambda_j^2}=\mathrm{tr}(Ae^{-tP})-\sum_{\lambda_j=0} \langle A\phi_j,\phi_j\rangle\ \ \ (t>0).
\end{equation}
Obviously, $\Pi(\max\{\frac{d+m}{2},0\},\infty)$ is a basic strip of $Mf$ which admits a meromorphic continuation to $\mathbb{C}$ with
$(Mf)(s)=\zeta(2s,A,P)\Gamma(s)$. Recall $\zeta(\cdot,A,P)$ is meromorphic on $\mathbb{C}$
whose only singularities are simple poles at $s=d+m-k$ ($k=0,1,2,\ldots$) with residues
$\mathscr{A}_k(A,P)$, while $\Gamma$ is meromorphic on $\mathbb{C}$  whose only singularities are simple poles at $s=-k$ ($k=0,1,2,\ldots$) with residues
$\frac{(-1)^k}{k!}$. Hence
to establish  (\ref{heat trace asymptotics}) via Lemma \ref{lemma 31} it suffices to prove
\begin{equation}\label{formula 39}
(Mf)(s)=O(|s|^{-2})\ \ \ (s\in\Pi(-n,n),\ |s|\rightarrow\infty)
\end{equation}
for all positive integers $n$.
 To this end we first assume  $s\in\Pi(-n,n)$ $(n\in\mathbb{N})$,
then let $h\geq n+2$ be a large enough integer so that one can use  (\ref{formula 37}) to get
\begin{equation}
(Mf)(s)=\zeta(2s,A,P)\Gamma(s)=\frac{(Mf_{2h})(s+h)}{\Gamma(s+h)}\Gamma(s)=(Mf_{2h})(s+h)\prod_{i=0}^{h-1}\frac{1}{s+i}.
\end{equation}
Finally by considering $s+h\in\Pi(2,n+h)$, (\ref{formula 39}) follows immediately from suitably applying (\ref{formula upper bound}) to
uniformly bound  $(Mf_{2h})(s+h)$ above for all $s\in\Pi(-n,n)$ with $|s|\rightarrow\infty$. This finishes the proof of (\ref{heat trace asymptotics}).

The precise relations between the mollified counting coefficients and some of the heat coefficients
as mentioned in the introductory part are left to the interested readers to verify.
For completeness we record these relations as a proposition for later use.
 We also remark that the above proof actually yields
 \begin{equation}\label{gamma}
\zeta(2s,A,P)\Gamma(s)=O(|s|^{-\infty})\ \ \ (s\in\Pi(-n,n),\ |s|\rightarrow\infty)
\end{equation}
for all positive integers $n$.

\begin{Proposition}\label{prop 23}
 If the order $m$ of $A$ is an integer, then
\begin{itemize}
\item $\mathscr{B}_k(A,P)=\frac{\Gamma(\frac{d+m-k}{2})}{2}\cdot\mathscr{A}_k(A,P)$\ \  ($d+m-k$ is positive or negative but odd);
\item $\mathscr{C}_k(A,P)=0$\ \ ($d+m+2k<0$);
\item $\mathscr{C}_k(A,P)=\frac{(-1)^{k+1}}{2\cdot k!}\cdot\mathscr{A}_{d+m+2k}(A,P)$\ \ ($d+m+2k\geq0$).
\end{itemize}
 If the order $m$ of $A$ is not an integer, then for all non-negative integers $k$:
\begin{itemize}
\item $\mathscr{B}_k(A,P)=\frac{\Gamma(\frac{d+m-k}{2})}{2}\cdot\mathscr{A}_k(A,P)$;
\item $\mathscr{C}_k(A,P)=0$.
\end{itemize}
\end{Proposition}

\begin{remark}
There exists an even stronger estimate than (\ref{gamma}) in \cite{Grubb Seeley 1996}. But to apply the inverse Mellin
transform theorem, a rough estimate like (\ref{gamma}) is sufficient.
\end{remark}

\subsection{Resolvent trace expansions} In this part we study the short time asymptotic expansion of the resolvent trace
$(f^{(N)}(t)\triangleq)$ $\mathrm{tr}(A(1+tP)^{-N/2})$,
where $N\in\mathbb{C}$ is  such that  $\mathrm{Re}(N)>\max\{d+m,0\}$.
Let  $M^{(N)}$ denote  the Mellin transform of
\[\mathrm{tr}(A(1+tP)^{-N/2})-\sum_{\lambda_j=0} \langle A\phi_j,\phi_j\rangle\ \ \ (t>0).\]
It is easy to verify that $M^{(N)}$ has a non-empty basic strip $\Pi(\max\{\frac{d+m}{2},0\},\frac{\mathrm{Re}(N)}{2})$
 in which
$
M^{(N)}(s)=\zeta(2s,A,P)B(s,\frac{N}{2}-s),
$
 where we recall the Beta function
 \[B(\alpha,\beta)=\int_0^1(1-t)^{\alpha-1}t^{\beta-1}dt=\int_0^{\infty}\frac{u^{\alpha-1}}{(1+u)^{\alpha+\beta}}du=\frac{\Gamma(\alpha)\Gamma(\beta)}{\Gamma(\alpha+\beta)}\ \ \ \]
is defined for $\mathrm{Re}(\alpha)>0$ and $\mathrm{Re}(\beta)>0$.
Hence $M^{(N)}$ admits a meromorphic continuation to $\Pi(-\infty,\frac{\mathrm{Re}(N)}{2})$ in which
\begin{equation}
M^{(N)}(s)=\zeta(2s,A,P)\frac{\Gamma(s)\Gamma(\frac{N}{2}-s)}{\Gamma(\frac{N}{2})}.
\end{equation}
For  any real $\beta$ in the strip $\Pi(\max\{\frac{d+m}{2},0\},\frac{\mathrm{Re}(N)}{2})$ and any positive integer $n$,
it is easy to see that
\[
\sup_{s\in\Pi(-n,\beta)}\big|\Gamma(\frac{N}{2}-s)\big|\leq\max_{\frac{\mathrm{Re}(N)}{2}-\beta \leq t\leq\frac{\mathrm{Re}(N)}{2}+n}\Gamma(t)<\infty.\]
Consequently by (\ref{formula 39}),
\begin{equation}
M^{(N)}(s)=O(|s|^{-2})\ \ \ (s\in\Pi(-n,\beta), |s|\rightarrow\infty),
\end{equation}
which immediately implies
\begin{equation}
f^{(N)}(t) \sim \sum_{k=0}^\infty \big( \mathscr{B}_{k}^{(N)}(A,P) t^{\frac{k-d-m}{2}} + \mathscr{C}_k^{(N)}(A,P) t^k \log(t) + \mathscr{D}_k^{(N)}(A,P) t^k\big)
\end{equation} via Lemma \ref{lemma 31} as the singularities of $\zeta(\cdot,A,P)$ and $\Gamma$ are known to be completely determined,
so is $M^{(N)}$. The precise relations between
the mollified counting coefficients and some of the resolvent trace coefficients are summarized in the following proposition:

\begin{Proposition}
 If the order $m$ of $A$ is an integer, then
\begin{itemize}
\item $\mathscr{B}_k^{(N)}(A,P)=\frac{\Gamma(\frac{d+m-k}{2})}{2}\cdot\frac{\Gamma(\frac{N+k-d-m}{2})}{\Gamma(\frac{N}{2})}\cdot\mathscr{A}_k(A,P)$\ \  ($d+m-k$ is positive or negative but odd);
\item $\mathscr{C}_k^{(N)}(A,P)=0$\ \ ($d+m+2k<0$);
\item $\mathscr{C}_k^{(N)}(A,P)=\frac{(-1)^{k+1}}{2\cdot k!}\cdot\frac{\Gamma(\frac{N}{2}+k)}{\Gamma(\frac{N}{2})}\cdot\mathscr{A}_{d+m+2k}(A,P)$\ \ ($d+m+2k\geq0$).
\end{itemize}
If the order $m$ of $A$ is not an integer, then for all non-negative integers $k$:
\begin{itemize}
\item $\mathscr{B}_k^{(N)}(A,P)=\frac{\Gamma(\frac{d+m-k}{2})}{2}\cdot\frac{\Gamma(\frac{N+k-d-m}{2})}{\Gamma(\frac{N}{2})}\cdot\mathscr{A}_k(A,P)$;
\item $\mathscr{C}_k^{(N)}(A,P)=0$.
\end{itemize}
\end{Proposition}

\subsection{Wodzicki residues}

This part is devoted to the proofs of (\ref{thm 13}) and (\ref{thm 14})
and their consequences. Both formulae are known and we state them here as two propositions for completeness.
Also, we refer the readers to the original references  for the proof of the crucial formula (\ref{Res C0}).

\begin{Proposition}\label{prop 28} $\mathscr{A}_k(A,P)=\mathrm{res}(AP^{\frac{k-d-m}{2}})$.
\end{Proposition}

\begin{proof}
 Since $AP^{\frac{k-d-m}{2}}$ is of integer order $k-d$ and $d+(k-d)+2\cdot0=k\geq0$, it follows first from Proposition \ref{prop 23} then from Proposition \ref{Prop 22} that
\[-2\mathscr{C}_0(AP^{\frac{k-d-m}{2}},P)=\mathscr{A}_k(AP^{\frac{k-d-m}{2}},P)=\mathscr{A}_k(A,P).\]
But according to (\ref{Res C0}) we also have
\[-2\mathscr{C}_0(AP^{\frac{k-d-m}{2}},P)=\mathrm{res}(AP^{\frac{k-d-m}{2}}).\]
Combining the above two formulae yields the desired result.
\end{proof}

\begin{Proposition}\label{prop 29} $\mathscr{C}_k(A,P)=\frac{(-1)^{k+1}}{2\cdot k!}\cdot\mathrm{res}(AP^k)$.
\end{Proposition}

\begin{proof}
If $A$ is of integer order $m\geq -d-2k$, then it follows first from Proposition \ref{prop 23} then from Proposition \ref{prop 28} that
\[\mathscr{C}_k(A,P)=\frac{(-1)^{k+1}}{2\cdot k!}\cdot\mathscr{A}_{d+m+2k}(A,P)=\frac{(-1)^{k+1}}{2\cdot k!}\cdot\mathrm{res}(AP^k).\]
If $A$ is not of integer order or $A$ is of integer order $m<-d-2k$, then it follows from definition that
$\mathrm{res}(AP^k)=0$ and from Proposition \ref{prop 23} that $\mathscr{C}_k(A,P)=0$. This finishes the proof.
\end{proof}

Next we study the property $\mathscr{A}_k(A,P)=0$ for certain operators $A\in\Psi(M;E)$.
We denote by $\Psi_{\mathrm{odd}}^{\mathbb{Z}}(M;E)$ the space of odd-class pseudodifferential operators
and by $\Psi_{\mathrm{even}}^{\mathbb{Z}}(M;E)$ the space of even-class pseudodifferential operators on sections of $E$, that is,
$A\in\Psi_{\mathrm{odd}}^{\mathbb{Z}}(M;E)$ if in local coordinates its symbol $\sum_{j\geq0}\sigma_{m-j}(A)$ $(m\in\mathbb{Z})$ satisfies
\begin{equation}
\sigma_{m-j}(A)(x,-\xi)=(-1)^{m-j}\sigma_{m-j}(A)(x,\xi)
\end{equation}
for all $x,\xi$ and $j$, while $A\in\Psi_{\mathrm{even}}^{\mathbb{Z}}(M;E)$ if
in local coordinates
\begin{equation}
\sigma_{m-j}(A)(x,-\xi)=(-1)^{m-j+1}\sigma_{m-j}(A)(x,\xi)
\end{equation}
for all $x,\xi$ and $j$. For example, any partial differential operator is odd-class while $P^{1/2}$ is even-class.
An operator $A\in\Psi_{\mathrm{odd}}^{\mathbb{Z}}(M;E)\cup\Psi_{\mathrm{even}}^{\mathbb{Z}}(M;E)$ is said of regular parity class if
its parity class agrees with that of $d$. It is easy to verify  that $\mathrm{res}(A)=0$ if $A$ is of
regular parity class.

\begin{Proposition} Let $A\in\Psi_{\mathrm{odd}}^{\mathbb{Z}}(M;E)\cup\Psi_{\mathrm{even}}^{\mathbb{Z}}(M;E)$.
If the parity class of $A$ agrees with that of $k-m$, then $\mathscr{A}_k(A,P)=0$.
\end{Proposition}

\begin{proof} We define a map $\tau$ by sending odd-class pseudodifferential operators to 0 and even-class ones to 1. Then (see \cite[Prop. 1.11]{Ponge})
$\tau(AP^{\frac{k-d-m}{2}})=\tau(A)+k-d-m$ $(\mathrm{mod}\ 2)$.
Thus if the parity class of $A$ agrees with that of $k-m$,
or equivalently, the parity class of $\tau(AP^{\frac{k-d-m}{2}})$ does not agree with that of $d$, or further equivalently, the parity class of  $AP^{\frac{k-d-m}{2}}$ agrees with that of $d$,
then $AP^{\frac{k-d-m}{2}}$ is of regular parity class and consequently by Proposition \ref{prop 28}, $\mathscr{A}_k(A,P)=\mathrm{res}(AP^{\frac{k-d-m}{2}})=0$.
\end{proof}

\begin{remark} The heat expansion (\ref{heat trace asymptotics}) can be written in a more rigorous way as
\begin{align}\label{trace 2}
 \mathrm{tr}(Ae^{-t P}) \sim \sum_{j\in\Lambda}\widetilde{\mathscr{B}}_j(A,P) t^{\frac{j-d-m}{2}}+\sum_{k=0}^\infty\big(\widetilde{\mathscr{C}}_k(A,P) t^k \log(t) + \widetilde{\mathscr{D}}_k(A,P) t^k\big),
\end{align}
where $\Lambda$ denotes the set of non-negative integers $j$ such that $\frac{j-d-m}{2}$ is not a non-negative integer.
In this notation system $\widetilde{\mathscr{B}}_j(A,P)$, $\widetilde{\mathscr{C}}_k(A,P)$ and $\widetilde{\mathscr{D}}_k(A,P)$ are uniquely determined.
According to Propositions \ref{prop 23}, \ref{prop 28}, \ref{prop 29},
$\widetilde{\mathscr{B}}_j(A,P)$ and $\widetilde{\mathscr{C}}_k(A,P)$ are certain Wodzicki residues.
In most situations, $\widetilde{\mathscr{D}}_k(A,P)$ are not Wodzicki residues but operator traces.
For the simplest example that $A$ is smoothing operator, one can (regard $A$ as of any real order and thus) deduce from (\ref{trace 2}) that
\begin{align}
 \mathrm{tr}(Ae^{-t P}) \sim \sum_{k=0}^\infty \widetilde{\mathscr{D}}_k(A,P) t^k.
\end{align}
In this situation it is so natural to expect $\mathrm{tr}(A)=\widetilde{\mathscr{D}}_0(A,P)$. Actually,
if any of the following three cases happens,
then $\widetilde{\mathscr{D}}_0(A,P)$ just is the Kontsevich-Vishik trace of $A$ which is independent of the choice of $P$
(see e.g. \cite{Grubb Schrohe,KV,Maniccia}):
\begin{itemize}
\item $m<-d$;
\item $m$ is not an integer;
\item $A$ is of regular parity class.
\end{itemize}
Since the Kontsevich-Vishik trace agrees with the standard $L^2$-operator trace on trace class operators,
the previous expectation is confirmed.
Furthermore, if $A=Q$ is a partial differential operator and if $d$ is even,
then one can express $\widetilde{\mathscr{D}}_0(A,P)$ as the Wodzicki residue of certain
pseudodifferential operator with log-polyhomogeneous symbol (see \cite{Mickelsson,Paycha,Scott 2005} for details).
\end{remark}


\subsection{Dirac operators}
This part is devoted to the proofs of (\ref{Mollified asmptotics Dirac new}), (\ref{117}) and (\ref{Ak AD}).
First, we introduce a classical pseudodifferential operator $B$ of order $m$ by defining
$B=A\frac{\mathrm{Sign}(D)+\mathrm{Id}_E}{2}$.
Since $\mathrm{supp}(N_A'-N_B')\subset(-\infty,0]$, we get (see also Subsection \ref{sub 2.2})
\[(\chi\ast N_A')(\mu)-(\chi\ast N_B')(\mu)=o(\mu^{-\infty})\ \ \ (\mu\rightarrow\infty).\]
By appealing to (\ref{Mollified asmptotics}) with $A,P,\lambda,\lambda_j$ replaced respectively by $B,D^2,\mu,|\mu_j|$, we obtain as $\mu\rightarrow\infty$ that
\begin{align*}
(\chi\ast N_B')(\mu)&=\sum_{j=1}^{\infty}\langle B\phi_j,\phi_j\rangle\chi(\mu-\mu_j)\\
&=\sum_{j=1}^{\infty}\langle B\phi_j,\phi_j\rangle\chi(\mu-|\mu_j|)\\
&\sim\sum_{k=0}^{\infty}\mathscr{A}_k(B,D^2)\mu^{d+m-k-1},
\end{align*}
which proves (\ref{Mollified asmptotics Dirac new}) as well as $\mathscr{A}_k(A,D)=\mathscr{A}_k(B,D^2)$ for all non-negative integers $k$.
Obviously by setting $A=F$, (\ref{117}) is a special case of $\mathscr{A}_k(A,D)=\mathscr{A}_k(B,D^2)$.
Now it follows from Proposition \ref{prop 28} that
\begin{equation}\label{219}
\mathscr{A}_k(A,D)=\mathrm{res}(A\frac{\mathrm{Sign}(D)+\mathrm{Id}_E}{2}|D|^{k-d-m})=\mathrm{res}(A\frac{D+|D|}{2}|D|^{k-d-m-1}),
\end{equation}
which proves (\ref{Ak AD}).

Next, we give an  equivalent expression for $\mathscr{A}_k(A,D)$.
 By Proposition \ref{Prop 22} one gets
\begin{equation}\label{220}
\mathscr{A}_k(A,D)=\mathscr{A}_k(B,D^2)=\frac{\mathscr{A}_k(A,D^2)+\mathscr{A}_k(AD,D^2)}{2}.
\end{equation}

Similar to the microlocalized spectral zeta function (\ref{zeta}), one can introduce the microlocalized spectral eta function (see e.g. \cite{Branson Gilkey,GSmith})  by defining
\begin{equation}
\eta(s,A,D)=\sum_{\mu_j\neq 0}\frac{\langle A\phi_j,\phi_j\rangle}{|\mu_j|^s}\cdot\mathrm{Sign}(\mu_j)\ \ \ (\mathrm{Re}(s)>d+m).
\end{equation}
But it is easy to see that \begin{equation}\eta(s,A,D)=\zeta(s,A\cdot\mathrm{Sign}(D),D^2).\end{equation}
So according to Proposition \ref{prop 28}, $\eta(s,A,D)$
admits a meromorphic continuation to $\mathbb{C}$ whose only singularities are simple poles at $s=d+m-k$ ($k=0,1,2,\ldots$) with residues
$\mathrm{res}(AD|D|^{k-d-m-1})$.

\subsection{Local counting functions}\label{sub 28} This part is devoted to the proof of (\ref{111}) and its consequences.
Recall $Q$ denotes a partial differential operator of order $m$
 acting  on $C^{\infty}(M;E)$.
Combining (\ref{kernel}) $\sim$ (\ref{heat kernel expansion}) yields
\begin{equation}\label{222}
\mathrm{tr}(Qe^{-tD^2})\sim\sum_{k=0}^{\infty}\int_M\mathrm{Tr}_E(\mathscr{H}_k(Q,D^2))\cdot t^{\frac{k-d-m}{2}}\ \ \ (t\rightarrow0^+),
\end{equation}
which compared with the heat expansion (\ref{heat trace asymptotics}) gives
  \begin{itemize}
\item $\mathscr{B}_k(Q,D^2)=\int_M\mathrm{Tr}_E(\mathscr{H}_k(Q,D^2))$\   ($d+m-k$ is positive or negative but odd);
\item ${\mathscr C}_k(Q,D^2)=0$.
\end{itemize}
According to Proposition \ref{prop 23} we also have
 \begin{itemize}
\item $\mathscr{B}_k(Q,D^2)=\frac{\Gamma(\frac{d+m-k}{2})}{2}\mathscr{A}_k(Q,D^2)$\   ($d+m-k$ is positive or negative but odd);
\item $\mathscr{C}_k(Q,D^2)=\frac{(-1)^{k+1}}{2\cdot k!}\mathscr{A}_{d+m+2k}(Q,D^2)$\ \ ($d+m+2k\geq0$).
\end{itemize}
Thus for all non-negative integers $k$,
\begin{equation}\label{223}
\mathscr{A}_k(Q,D^2)=\frac{2}{\Gamma(\frac{d+m-k}{2})}\int_M\mathrm{Tr}_E(\mathscr{H}_k(Q,D^2)).
\end{equation}
By combining (\ref{17}), (\ref{220}), (\ref{223}) and noting $\mathscr{H}_k(FQ,D^2)=F\mathscr{H}_k(Q,D^2)$, we get
\begin{align}
\int_M\mathrm{Tr}_E(F\mathscr{L}_k(D))&=\mathscr{A}_k(F,D)
=\frac{\mathscr{A}_k(F,D^2)+\mathscr{A}_k(FD,D^2)}{2}\label{224}\\
&=\int_M\mathrm{Tr}_E\Big(F\big(\frac{\mathscr{H}_k(\mathrm{Id}_E,D^2)}{\Gamma(\frac{d-k}{2})}
+\frac{\mathscr{H}_k(D,D^2)}{\Gamma(\frac{d+1-k}{2})}\big)\Big),\nonumber
\end{align}
which proves (\ref{111}) as $F$ can be any smooth endomorphism of $E$. To be precise, one can set $F$ to be the adjoint of
\[\mathscr{L}_k(D)-\big(\frac{\mathscr{H}_k(\mathrm{Id}_E,D^2)}{\Gamma(\frac{d-k}{2})}
+\frac{\mathscr{H}_k(D,D^2)}{\Gamma(\frac{d+1-k}{2})}\big)\]
then apply (\ref{224}) to conclude (\ref{111}).

Next, we study the local counting function of Laplace type operators but first have a look at that of
Dirac type operators.
According to (\ref{16}),
\[
\sum_{j=1}^{\infty}\Phi_j\chi(\mu\mp\mu_j)\sim\ \sum_{k=0}^{\infty}\mu^{d-k-1}\mathscr{L}_k(\pm D)\ \ \ (\mu\rightarrow\infty),
\]
which easily implies (see also \cite{Ivrii})
\[
\sum_{j=1}^{\infty}\Phi_j\chi(\mu-|\mu_j|)\sim\ \sum_{k=0}^{\infty}\mu^{d-k-1}(\mathscr{L}_k(D)+\mathscr{L}_k(-D))\ \ \ (\mu\rightarrow\infty).
\]
As an application of (\ref{111}), one gets $\mathscr{L}_k(D)+\mathscr{L}_k(-D)=\frac{2\mathscr{H}_k(\mathrm{Id}_E,D^2)}{\Gamma(\frac{d-k}{2})}$. Thus
\begin{equation}
\sum_{j=1}^{\infty}\Phi_j\chi(\mu-|\mu_j|)\sim\ \sum_{k=0}^{\infty}\mu^{d-k-1}\frac{2\mathscr{H}_k(\mathrm{Id}_E,D^2)}{\Gamma(\frac{d-k}{2})}\ \ \ (\mu\rightarrow\infty).
 \end{equation}
Generally for Laplace type operators $P$, one can similarly establish that
\begin{equation}
\sum_{j=1}^{\infty}\Phi_j\chi(\lambda-\lambda_j)\sim\ \sum_{k=0}^{\infty}\lambda^{d-k-1}\frac{2\mathscr{H}_k(\mathrm{Id}_E,P)}{\Gamma(\frac{d-k}{2})}\ \ \ (\lambda\rightarrow\infty).
 \end{equation}

\section{The Bochner-Weitzenb\"{o}ck technique and a direct proof of Theorem \ref{theorem 1.6}}\label{Section 3}

In this section we will introduce some basic notations about Dirac operators and
give a direct proof of Theorem \ref{theorem 1.6} by employing the Bochner-Weitzenb\"{o}ck technique  in
 Riemannian geometry (see e.g. \cite{BGV,Gilkey 1975,Gilkey AFSG,Lawson}).

\subsection{Bochner-Weitzenb\"{o}ck technique}

We begin by reviewing the concept of the Dirac bundle   $(E,\gamma,\widetilde{\nabla})$ (see e.g. \cite{Lawson}),
where $\gamma$ is a Clifford module structure on $E$ and $\widetilde{\nabla}$ is a connection on $E$ that is compatible with $\gamma$.

First, let $\mathrm{Cl}(TM)$ be the universal unital complex algebra bundle generated
by the tangent bundle $TM$ subject to the commutation relation
$X\ast Y+Y\ast X=-2g({X,Y})$,
where $X,Y\in C^{\infty}(M;TM)$ and $\ast$ denotes the algebra operation.
A Clifford module structure on  $E$ just is a unital algebra morphism
$\gamma:\mathrm{Cl}(TM)\rightarrow\mbox{End}(E)$.
Given a connection $\widehat{\nabla}$ on the Clifford module $E=(E,\gamma)$,
$
\gamma\widehat{\nabla}\triangleq\gamma(e_k)\widehat{\nabla}_{e_k}
$
is a well-defined first order partial differential operator acting on $C^{\infty}(M;E)$, where $\{e_k\}_{k=1}^d$
denotes any orthonormal basis of $T_xM$ at each point $x\in M$.
Next, one can naturally extend the Riemannian metric $\mu_g$ and connection $\nabla$ on $TM$ to $\mathrm{Cl}(TM)$
 with the properties that:
 the extended connection denoted also by $\nabla$ preserves the extended metric, and
\begin{equation}\label{clifford preverse}
\nabla_{X}(\alpha\ast\beta)=(\nabla_{X}\alpha)\ast\beta+\alpha\ast(\nabla_{X}\beta)
\end{equation}
for all $X\in C^{\infty}(M;TM)$, $\alpha,\beta\in C^{\infty}(M;\mathrm{Cl}(TM))$. We then call
a connection $\widetilde{\nabla}$ on the Clifford module $(E,\gamma)$  compatible with $\gamma$  if
\begin{equation}\label{compatible original}
\widetilde{\nabla}_X(\gamma(\alpha)\phi)=\gamma(\nabla_X\alpha)\phi+\gamma(\alpha)(\widetilde{\nabla}_X\phi)
\end{equation}
for all  $X\in C^{\infty}(M;TM)$, $\alpha\in  C^{\infty}(M;\mathrm{Cl}(TM))$,  $\phi\in  C^{\infty}(M;E)$.
It is well-known (\cite{Branson Gilkey}) that on
a Clifford module there always exists  a compatible connection. Such a triple $(E,\gamma,\widetilde{\nabla})$ is called a Dirac bundle.
With the concept of Dirac bundle available, it is easy to see that any operator $D$ of Dirac type acting on $C^{\infty}(M;E)$ can always be written as
$
D=\gamma\widetilde{\nabla}+\psi,
$
where $(E,\gamma,\widetilde{\nabla})$ is some Dirac bundle with $\gamma$ uniquely determined by the principal symbol $\sigma_D$ of $D$,
 $\psi\in C^{\infty}(M;\mathrm{End}(E))$. We call $\psi$ the potential of $D$ associated with the Dirac bundle $(E,\gamma,\widetilde{\nabla})$.

Let $(E,\gamma,\widetilde{\nabla})$ be a Dirac bundle and let
$D_{\psi}=\gamma\widetilde{\nabla}+\psi$ be a Dirac type operator with potential $\psi$.
The Bochner-Weitzenb\"{o}ck technique  permits us to find a unique connection $\nabla_{\psi}$ on $E$ and a unique  $V_{\psi}\in C^{\infty}(M;\mathrm{End}(E))$ such that
 $D_{\psi}^2=\Delta^{\nabla_{\psi}}+V_{\psi}$,
 where $\Delta^{\nabla_{\psi}}=\nabla_{\psi}^{\ast}\nabla_{\psi}$ denotes the connection Laplacian generated by $\nabla_{\psi}$.
In a series of papers by Gilkey and his collaborators (see e.g. \cite{Branson Gilkey,Gilkey AFSG,Gilkey 2007}),
 $V_{\psi}$ is always written as a function explicitly  depending on  $\nabla_{\psi}$ and a few others.
 Our feature of proving Theorem \ref{theorem 1.6}  is to first
 express $V_{\psi}$ as a function of  $\gamma$, $\widetilde{\nabla}$,
 $\psi$, then treat the first two variables as dummy ones.


Given a Dirac bundle $(E,\gamma,\widetilde{\nabla})$, we define a connection $\overline{\nabla}$ on $\mathrm{End}(E)$ by
\begin{equation}\label{new connection}
(\overline{\nabla}_X\omega)(\phi)=\widetilde{\nabla}_X(\omega(\phi))-\omega(\widetilde{\nabla}_X\phi)
\end{equation}
for all $X\in C^{\infty}(M;TM)$,  $\omega\in C^{\infty}(M;\mathrm{End}(E))$, $\phi\in C^{\infty}(M;E)$. It is easy to verify that $\overline{\nabla}$ is
compatible with  $\gamma$ in the following sense:
\begin{align}
(\overline{\nabla}_X)(\omega\gamma(Y))&=(\overline{\nabla}_X\omega)\gamma(Y)+\omega\gamma(\nabla_XY)\ \  \label{compatible 52}\\
(\overline{\nabla}_X)(\gamma(Y)\omega)&=\gamma(Y)(\overline{\nabla}_X\omega)+\gamma(\nabla_XY)\omega\ \ \label{compatible 53}
\end{align}
 where $X,\omega,\phi$ remain the same meanings as before, $Y\in C^{\infty}(M;TM)$.

We also need to fix a few notations.  If $(x^1,\ldots,x^d)$ is a system of local coordinates on $M$, then let
$\{\partial_i=\frac{\partial}{\partial x^i}\}$ denote the local  coordinate frames of $TM$.
We use the metric tensor to identify the tangent and cotangent bundles $TM=T^{*}M$, which means locally  $dx^j\equiv g^{ij}\partial_i$,
where $(g^{ij})$ is the inverse of the matrix $(g_{ij}\triangleq g(\partial_i,\partial_j))$.
The curvature tensor of any connection $\widehat{\nabla}$ on  $E$ is denoted by $R^{\widehat{\nabla}}$, that is,
   $R_{XY}^{\widehat{\nabla}}=[\widehat{\nabla}_X,\widehat{\nabla}_Y]-\widehat{\nabla}_{[X,Y]}$
   for  $X,Y\in C^{\infty}(M;TM)$.  Let $\Gamma_{ij}^k$ denote the Christoffel symbols
of the Riemannian connection $\nabla$, that is, $\nabla_{\partial_i}\partial_j=\Gamma_{ij}^k\partial_k$.
The connection Laplacian $\triangle^{\widehat{\nabla}}$  locally is of the form $-g^{ij}(\widehat{\nabla}_{\partial_i}\widehat{\nabla}_{\partial_j}-\Gamma_{ij}^k\widehat{\nabla}_{\partial_k})$.
Finally, we denote  $U^{i_1,\ldots,i_n}=U(dx^{i_1},\ldots,dx^{i_n})$ and
$U_{i_1,\ldots,i_n}=U(\partial_{i_1},\ldots,\partial_{i_n})$
for any map $U$ defined on the $n$-fold product of $TM$. 


\begin{Proposition}\label{lemma 41}
Let
$D$ be a Dirac type operator of potential $\psi$ associated with the Dirac bundle $(E,\gamma,\widetilde{\nabla})$.
  Let $L=L_{\psi}:TM\rightarrow\mathrm{End}(E)$ denote the  map  defined by
  \begin{equation}\label{555}L(X)=\frac{\gamma(X)\psi+\psi\gamma(X)}{2}.\end{equation}
  Then $\nabla_{\psi}\triangleq\widetilde{\nabla}-L$  is a connection on $E$ and
  $V_{\psi}\triangleq D^2-\Delta^{\nabla_{\psi}}\in C^{\infty}(M;\mathrm{End}(E))$.
  Locally we have
  \begin{equation}
   V_{\psi}=\frac{1}{2}\gamma^i\gamma^jR_{ij}^{\widetilde{\nabla}}+
   \frac{1}{2}[\gamma^i,\overline{\nabla}_{i}\psi]+L^iL_i
   +\psi^2.
  \end{equation}
\end{Proposition}

\begin{proof}
Let
$L$ be naturally identified with  \[C^{\infty}(M;TM)\times C^{\infty}(M;E)\stackrel{L}{\rightarrow}  C^{\infty}(M;E),\] which is
$C^{\infty}(M)$-linear in both variables. Thus
$\widehat{\nabla}\triangleq\widetilde{\nabla}-L$ is a connection on $E$. According to (\ref{new connection}) we have
\begin{align*}
\triangle^{\widetilde{\nabla}}-\triangle^{\widehat{\nabla}}&=
-g^{ij}(\widetilde{\nabla}_i\widetilde{\nabla}_j-\widetilde{\nabla}_i\widehat{\nabla}_j+\widetilde{\nabla}_i\widehat{\nabla}_j-\widehat{\nabla}_i\widehat{\nabla}_j)
+g^{ij}\Gamma_{ij}^k(\widetilde{\nabla}_k-\widehat{\nabla}_k)\\
&=-g^{ij}\widetilde{\nabla}_iL_j-g^{ij}L_i\widehat{\nabla}_j+g^{ij}\Gamma_{ij}^kL_k\\
&=-g^{ij}((\overline{\nabla}_iL_j)+L_j\widetilde{\nabla}_i)-g^{ij}L_i(\widetilde{\nabla}_j-L_j)+g^{ij}\Gamma_{ij}^kL_k\\
&=-2L^i\widetilde{\nabla}_i-g^{ij}(\overline{\nabla}_iL_j)+L^iL_i+g^{ij}\Gamma_{ij}^kL_k.
\end{align*}
On the other hand, according to  (\ref{compatible original}) and (\ref{new connection})
we have
\begin{align*}
D^2&=(\gamma^i\widetilde{\nabla}_i+\psi)(\gamma^j\widetilde{\nabla}_j+\psi)\\
&=\gamma^i\gamma^j\widetilde{\nabla}_i\widetilde{\nabla}_j-\Gamma_{ik}^j\gamma^i\gamma^k\widetilde{\nabla}_j+
\gamma^i(\overline{\nabla}_i\psi)+\gamma^i\psi\widetilde{\nabla}_i+\psi\gamma^j\widetilde{\nabla}_j+\psi^2\\
&=\gamma^i\gamma^j\frac{\widetilde{\nabla}_i\widetilde{\nabla}_j+\widetilde{\nabla}_j\widetilde{\nabla}_i}{2}+
\gamma^i\gamma^j\frac{R_{ij}^{\widetilde{\nabla}}}{2}+\Gamma_{ik}^jg^{ik}\widetilde{\nabla}_j +
\gamma^i(\overline{\nabla}_i\psi)+(\gamma^i\psi+\psi\gamma^i)\widetilde{\nabla}_i+\psi^2\\
&=-g^{ij}\widetilde{\nabla}_i\widetilde{\nabla}_j+g^{ij}\Gamma_{ij}^k\widetilde{\nabla}_k+\gamma^i\gamma^j\frac{R_{ij}^{\widetilde{\nabla}}}{2}
+
\gamma^i(\overline{\nabla}_i\psi)+(\gamma^i\psi+\psi\gamma^i)\widetilde{\nabla}_i+\psi^2\\
&=\triangle^{\widetilde{\nabla}}+(\gamma^i\psi+\psi\gamma^i)\widetilde{\nabla}_i+\gamma^i\gamma^j\frac{R_{ij}^{\widetilde{\nabla}}}{2}
+
\gamma^i(\overline{\nabla}_i\psi)+\psi^2.
\end{align*}
Consequently, combining the above calculations yields
\[D^2=\triangle^{\widehat{\nabla}}+\gamma^i\gamma^j\frac{R_{ij}^{\widetilde{\nabla}}}{2}
+
\gamma^i(\overline{\nabla}_i\psi)+\psi^2-g^{ij}(\overline{\nabla}_iL_j)+L^iL_i+g^{ij}\Gamma_{ij}^kL_k\triangleq \triangle^{\widehat{\nabla}}+V_{\psi}.\]
Next let us simplify the expression of $V_{\psi}$.
According to   (\ref{compatible 52}) and (\ref{compatible 53}) we have
\begin{align*}
-g^{ij}(\overline{\nabla}_iL_j)&=-g^{ij}\frac{\gamma_j(\overline{\nabla}_i\psi)+\Gamma_{ij}^k\gamma_k\psi+(\overline{\nabla}_i\psi)\gamma_j+
\Gamma_{ij}^k\psi\gamma_k}{2}\\
&=-g^{ij}\frac{\gamma_j(\overline{\nabla}_i\psi)+(\overline{\nabla}_i\psi)\gamma_j
}{2}-g^{ij}\Gamma_{ij}^kL_k\\
&=-\frac{\gamma^i(\overline{\nabla}_i\psi)+(\overline{\nabla}_i\psi)\gamma^i}{2}-g^{ij}\Gamma_{ij}^kL_k,
\end{align*}
which implies
\begin{align*}V_{\psi}&=\gamma^i\gamma^j\frac{R_{ij}^{\widetilde{\nabla}}}{2}
+
\gamma^i(\overline{\nabla}_i\psi)+\psi^2-\frac{\gamma^i(\overline{\nabla}_i\psi)+(\overline{\nabla}_i\psi)\gamma^i}{2}+L^iL_i\\
&=\gamma^i\gamma^j\frac{R_{ij}^{\widetilde{\nabla}}}{2}
+
\frac{1}{2}[\gamma^i,\overline{\nabla}_i\psi]+\psi^2+L^iL_i.\end{align*}
This  finishes the proof of Proposition \ref{lemma 41}.
\end{proof}

\subsection{Direct proof of Theorem \ref{theorem 1.6}}

 Let $D_{\epsilon}$ denote the Dirac type operator of potential $\psi_{\epsilon}=\psi-\epsilon F$ associated with the Dirac bundle $(E,\gamma,\widetilde{\nabla})$, where $\epsilon$ is a real parameter. In particular $D=D_0$.
Let $\nabla_{\epsilon}=\nabla_{\psi_{\epsilon}}$, $V_{\epsilon}=V_{\psi_{\epsilon}}$ be as determined by Proposition \ref{lemma 41}. Locally,
  \[V_{\epsilon}=\frac{1}{2}\gamma^i\gamma^jR_{ij}^{\widetilde{\nabla}}+
   \frac{1}{2}[\gamma^i,\overline{\nabla}_{i}\psi_{\epsilon}]+(L_{\psi_{\epsilon}})^i(L_{\psi_{\epsilon}})_i
   +\psi_{\epsilon}^2.\]
Thus
\[\frac{d}{d\epsilon}\Big|_{\epsilon=0}\mathrm{Tr}(V_{\epsilon})=
\mathrm{Tr}\big((L_{-F})^i(L_{\psi})_i+(L_{\psi})^i(L_{-F})_i\big)-2\mathrm{Tr}(F\psi),
\]
which following simplification via the facts $\gamma^i\gamma_i=-d$ and  $\widehat{\psi}=\gamma^i\psi\gamma_i=\gamma_i\psi\gamma^i$ gives
\begin{equation}\label{diff}\frac{d}{d\epsilon}\Big|_{\epsilon=0}\mathrm{Tr}(V_{\epsilon})=\mathrm{Tr}\big(F((d-2)\psi-\widehat{\psi})\big).
\end{equation}
But it is known (see e.g. \cite[Lemma 2.1]{Branson Gilkey}, \cite[Thm. 3.3.1]{Gilkey AFSG}) that
\begin{equation}\label{49}
\int_M\mathrm{Tr}(\mathscr{H}_1(FD,D^2))=\frac{1}{12\cdot(4\pi)^{d/2}}\cdot
\frac{d}{d\epsilon}\Big|_{\epsilon=0}\int_M\mathrm{Tr}(\tau\mathrm{Id}_E-6V_{\epsilon}),\end{equation}
where $\tau$ is the scalar curvature of $M$. By combining (\ref{diff}), (\ref{49}) and by considering the fact
$\mathscr{H}_1(FD,D^2)=F\mathscr{H}_1(D,D^2)$, one gets
\[
\int_M\mathrm{Tr}\big(F\cdot\mathscr{H}_1(D,D^2)\big)=
\frac{1}{(4\pi)^{d/2}}\cdot \int_M\mathrm{Tr}\big(F\cdot\frac{\widehat{\psi}-(d-2)\psi}{2}\big).
\]
Consequently,
\begin{equation}\label{H1}
\mathscr{H}_1(D,D^2)=
\frac{1}{(4\pi)^{d/2}}\cdot \frac{\widehat{\psi}-(d-2)\psi}{2}.
\end{equation}
as $F$ can be any smooth endomorphism of $E$.
By (\ref{17}) and (\ref{111}) (see also (\ref{224})),
\[
\mathscr{A}_1(F,D)
=\int_M\mathrm{Tr}\Big(F\big(\frac{\mathscr{H}_1(\mathrm{Id}_E,D^2)}{\Gamma(\frac{d-1}{2})}
+\frac{\mathscr{H}_1(D,D^2)}{\Gamma(\frac{d}{2})}\big)\Big).
\]
But  $\mathscr{H}_1(\mathrm{Id}_E,D^2)=0$ (see e.g. \cite{Gilkey Invariance}), which completes the proof of Theorem \ref{theorem 1.6}.

\begin{remark}\label{remark self adjoint}
It is known (see e.g. \cite{Gilkey Invariance,Seeley 1966}) that
(\ref{kernel}) and (\ref{heat kernel expansion}) still hold even if without assuming the self-adjointness of $D$ if the integral kernels are appropriately interpreted. In this general situation our direct proof of Theorem \ref{theorem 1.6} can also yield (\ref{H1}) with no modification at all.
\end{remark}

\section{Mollified counting coefficients II}\label{section 4}

In this section we will express $\mathscr{A}_1(A,P)$ in terms of principal and sub-principal symbols of $A$ and $P$, and prove Theorems \ref{theorem 1.62}, \ref{theorem 1.6}. 
We refer the readers to \cite{Chervova JST,ChervovaDV,Duistermaat Guillemin 1975,Fang,Ivrii,Ivrii 2,SafarovV,Sandoval}
for some other explicit two-term asymptotic expansions in various situations. 

We begin by reviewing
the definition of the sub-principal symbol of a classical pseudodifferential operator.

Recall $M$ is a closed manifold over which
the space of smooth sections of the half-density bundle is denoted by $C^{\infty}(M,\Omega_{1/2})$.
An operator $H:C^{\infty}(M,\Omega_{1/2})\rightarrow C^{\infty}(M,\Omega_{1/2})$
is called a classical pseudodifferential operator of order $m$
if on every coordinate patch its total symbol admits an asymptotic expansion
$\sum_{j=0}^{\infty}\sigma_H^{(j)}(x,\xi)$ with $\sigma_H^{(j)}$ homogeneous of degree $m-j$ in $\xi$.
 It is well-known that the sub-principal symbol (see e.g. \cite[Section 5.2]{Duistermaat-Hormander})  of $H$
defined in local coordinates  by \begin{equation}
\mathrm{Sub}(H)=\sigma_H^{(1)}+\frac{\mathrm{i}}{2}\cdot\frac{\partial^2\sigma_H^{(0)}}{\partial x^k\partial\xi_k},
\end{equation}
transforms like a homogeneous smooth function of degree $m-1$ on $T^{\ast}M\backslash0$ under change of charts.

Let $A:C^{\infty}(M)\rightarrow C^{\infty}(M)$ be a classical pseudodifferential operator of order $m$. In local coordinates we denote by
$\sigma_A^{(j)}=\sigma_A^{(j)}(x,\xi)$  ($j=0,1,2,\ldots$) the homogeneous part of degree $m-j$ of the total symbol of $A$\footnote{Note we also denote $\sigma_A^{(j)}$ by $\sigma_{m-j}(A)$ in the
previous sections.}.
Given a Riemannian metric $g$ on $M$ one can identify  $C^{\infty}(M)$ with $C^{\infty}(M,\Omega_{1/2})$
via the  map
\[f\in C^{\infty}(M)\stackrel{\Phi}{\longrightarrow}f\sqrt{\mu_g}\in C^{\infty}(M,\Omega_{1/2}),\]
where $\mu_g$ is the Riemannian density (called also metric measure) on $(M,g)$. Then $H=\Phi\circ A\circ \Phi^{-1}$
is a classical pseudodifferential operator  of order $m$ on half-densities.
On local coordinates let $G\triangleq\mathrm{det}(g_{ij})$. Note
$H=G^{1/4}AG^{-1/4}$
can be regarded as the product of three (local) pseudodifferential operators with corresponding
total symbols $G(x)^{1/4}$, $\sum_{j=0}^{\infty}\sigma_A^{(j)}(x,\xi)$,  $G(x)^{-1/4}$. Thus one can deduce from the product rule of pseudodifferential operators that
 \[\mathrm{Sub}(H)=\sigma_A^{(1)}+\frac{\mathrm{i}}{2}\cdot\frac{\partial^2\sigma_A^{(0)}}{\partial x^k\partial\xi_k}+
\frac{\mathrm{i}}{2}\cdot \frac{\partial\sigma_A^{(0)}}{\partial{\xi_k}}\cdot\frac{\partial (\log\sqrt{G})}{\partial x^k}\]
on local coordinates. This means the sub-principal symbol of $A$ locally defined  by (\cite[Remark 2.1.10]{SafarovV})
\begin{equation}\label{aaa}\mathrm{Sub}(A)=\sigma_A^{(1)}+\frac{\mathrm{i}}{2}\cdot\frac{\partial^2\sigma_A^{(0)}}{\partial x^k\partial\xi_k}+
\frac{\mathrm{i}}{2}\cdot \frac{\partial\sigma_A^{(0)}}{\partial{\xi_k}}\cdot\frac{\partial (\log\sqrt{G})}{\partial x^k},\end{equation}
is a homogeneous smooth function of degree $m-1$ on $T^{\ast}M\backslash0$.

Now let $A:C^{\infty}(M;E)\rightarrow C^{\infty}(M;E)$ be a classical pseudodifferential operator of order $m$.
For a fixed local bundle trivialization of $E$ over a coordinate neighborhood $U\subset M$, one can naturally
identify $A^U:C_c^{\infty}(U;E)\rightarrow C^{\infty}(U;E)$, the restriction of $A$ onto $U$,
 with a matrix $(A^U_{\mu\nu})_{1\leq\mu,\nu\leq r}$ of classical pseudodifferential operators $A^U_{\mu\nu}:C^{\infty}(M)\rightarrow C^{\infty}(M)$
 of order $m$, where $r$ denotes the rank of $E$. This implies that the sub-principal symbol of $A$
 defined by (see also \cite{Jakobson})
$
 \mathrm{Sub}(A)=(\mathrm{Sub}(A^U_{\mu\nu}))_{1\leq\mu,\nu\leq r}
$
 is a homogeneous smooth $\mathrm{Mat}(r,\mathbb{C})$-valued  function of degree $m-1$ on $T^{\ast}U\backslash0$.
 On local coordinates $\mathrm{Sub}(A)$ is still of the form (\ref{aaa}). We should remember, however, that
 the definition of the sub-principal symbol depends on the choice of local frames for $E$ and Riemannian metric on $M$.
 If the principal symbol is a multiple of the identity then the sub-principal symbol can be invariantly defined as a partial connection
 along the Hamiltonian vector field generated by the principal symbol (see \cite{Jakobson}, Section 3.1).

\begin{remark}
Let $M$ be a closed manifold and let $A:C^{\infty}(M)\rightarrow C^{\infty}(M)$ be a classical pseudodifferential operator of order $m$.
Some authors (\cite{SafarovV}) like to identify $C^{\infty}(M)$ with $C^{\infty}(M,\Omega_{1/2})$ by first introducing
a smooth positive density $\rho$ on $M$ then identifying functions with half-densities by $f\leftrightarrow f\sqrt{\rho}$.
Suppose this is the case then the sub-principal symbol of $A$  locally defined by (\cite[Remark 2.1.10]{SafarovV})
\begin{equation}\label{bbb}\mathrm{Sub}(A)=\sigma_A^{(1)}+\frac{\mathrm{i}}{2}\cdot\frac{\partial^2\sigma_A^{(0)}}{\partial x^k\partial\xi_k}+
\frac{\mathrm{i}}{2}\cdot \frac{\partial\sigma_A^{(0)}}{\partial{\xi_k}}\cdot\frac{\partial (\log \rho)}{\partial x^k},\end{equation}
is a homogeneous smooth function of degree $m-1$ on $T^{\ast}M\backslash0$. We point out there is no essential difference
between (\ref{aaa}) and (\ref{bbb}) as it is always possible to endow $M$ with a Riemannian metric $g$ such that $\mu_g=\rho$.
To see this one can first endow $M$ with an arbitrary Riemannian metric $\widehat{g}$,
 then get a positive smooth function $\kappa$ on $M$
so that $\rho=\kappa\mu_{\widehat{g}}$, and finally define $g=\kappa^{2/d}\widehat{g}$, where $d$ is the dimension of $M$.
\end{remark}

\begin{example}
Let $P=\Delta^{\nabla}$ denote the connection Laplacian generated by the connection $\nabla$ on a vector bundle $E$ of rank $r$.
For fixed local frame $\{s_{\mu}\}_{\mu=1}^r$ for $E|_U$ there exists a matrix
$\omega=(\omega_{\mu\nu})_{1\leq \mu,\nu\leq r}$ of one-forms  on $U$ such that
$\nabla s_{\mu}=\omega_{\mu\nu}\otimes s_{\nu}
$. Note in any local coordinates system $(x^1,\ldots,x^d)$ on $U$, $P=-g^{ij}(\nabla_i\nabla_j-\Gamma_{ij}^k\nabla_k)$.
Letting $\nabla_j=\partial_j+b_j$ one can easily get $\sigma_P^{(0)}=g^{jk}(x)\xi_j\xi_k$, $\sigma_P^{(1)}=-2\mathrm{i}g^{jk}b_j\xi_k+\mathrm{i}g^{ln}\Gamma_{ln}^k\xi_k$.
Considering $\frac{\partial (\log\sqrt{G})}{\partial x^k}=\Gamma_{kn}^n$ (see e.g. \cite[Prop. 2.8]{Chavel}, \cite[Sec. 2.5]{Chen}, \cite[Sec. 6]{ChervovaDV}) we have
\begin{align*}
\mathrm{Sub}(P)&=\sigma_P^{(1)}+\frac{\mathrm{i}}{2}\cdot\frac{\partial^2\sigma_P^{(0)}}{\partial x^k\partial\xi_k}+
\frac{\mathrm{i}}{2}\cdot \frac{\partial\sigma_P^{(0)}}{\partial{\xi_k}}\cdot\frac{\partial (\log\sqrt{G})}{\partial x^k}\\
&=\sigma_P^{(1)}+\mathrm{i}\cdot\frac{\partial g^{jk}}{\partial x^k}\cdot\xi_j+\mathrm{i}\cdot g^{jk}\xi_j\cdot\Gamma_{kn}^n\\
&=(-2\mathrm{i}g^{jk}b_j\xi_k+\mathrm{i}g^{ln}\Gamma_{ln}^k\xi_k)+\mathrm{i}\cdot(-\Gamma_{kn}^jg^{nk}-\Gamma_{kn}^kg^{nj})\cdot\xi_j+\mathrm{i}\cdot g^{jk}\xi_j\cdot\Gamma_{kn}^n\\
&=-2\mathrm{i}g^{jk}b_j\xi_k.
\end{align*}
Also, it is easy to check that $b_j=\omega^T(\partial_j)$ where $\omega^T$ denotes the transpose of $\omega$. Thus in an invariant manner,
$
\mathrm{Sub}(P)(x,\xi)=-2\mathrm{i}g(dx^j,dx^k)\omega^T(\partial_j)\xi_k
=-2\mathrm{i}g(\omega^T,\xi).
$
\end{example}

\begin{example}\label{Example Dirac}
Let $D=\gamma\nabla$ be the associated generalized Dirac operator of a Dirac bundle $(E,\gamma,\nabla)$ of rank $r$.
We adopt all the notations about $\nabla$ used in the previous example. Then $\sigma_D^{(0)}=\mathrm{i}\gamma^j\xi_j$, $\sigma_D^{(1)}=\gamma^{j}b_j$, and consequently,
\begin{align*}
\mathrm{Sub}(D)&=\sigma_D^{(1)}+\frac{\mathrm{i}}{2}\cdot\frac{\partial^2\sigma_D^{(0)}}{\partial x^k\partial\xi_k}+
\frac{\mathrm{i}}{2}\cdot \frac{\partial\sigma_D^{(0)}}{\partial{\xi_k}}\cdot\frac{\partial (\log\sqrt{G})}{\partial x^k}\\
&=\gamma^{j}b_j-\frac{1}{2}\cdot\frac{\partial \gamma^k}{\partial x^k}-\frac{1}{2}\cdot\gamma^k\cdot\Gamma_{kn}^n\\
&=\gamma^{j}b_j-\frac{1}{2}\cdot([\gamma^k,b_k]-\Gamma_{kn}^k\gamma^n)-\frac{1}{2}\cdot\gamma^k\cdot\Gamma_{kn}^n\\
&=\frac{\gamma^kb_k+b_k\gamma^k}{2},
\end{align*}
where the third equality follows from (\ref{compatible original}) (see also the next subsection). On each inner product space $(T_xM,g_x)$, we introduce two bilinear maps $J_x$, $K_x$ sending $X_x,Y_x\in T_xM$ respectively to
$\gamma(X_x)\omega^T(Y_x)$ and $\omega^T(X_x)\gamma(Y_x)$. Then it is easy to see that
\begin{equation}
\mathrm{Sub}(D)(x,\xi)=\frac{\mathrm{Tr}(J_x)+\mathrm{Tr}(K_x)}{2},
\end{equation}
which means $\mathrm{Sub}(D)(x,\xi)$ actually is independent of $\xi$.
\end{example}

With the sub-principal symbol concept available, we can express certain Wodzicki residues in a more invariant way.
A key ingredient (see e.g. \cite{Fedosov,Sitarz}) is that if $f$ is a smooth homogeneous function  of degree $1-d$ on $\mathbb{R}^d\backslash\{0\}$, then
 $\int_{|\xi|=1}\frac{\partial f}{\partial\xi_k}dS(\xi)=0$ for each $k$. For example, if
  $A:C^{\infty}(M;E)\rightarrow C^{\infty}(M;E)$ is a classical pseudodifferential operator of order $1-d$,
  then it is easy to see that
  \[\int_{|\xi|=1}\mathrm{Tr}(\sigma_A^{(1)}(x,\xi))dS(\xi)=\int_{|\xi|=1}\mathrm{Tr}(\mathrm{Sub}(A)(x,\xi))dS(\xi),\]
  and consequently, the global density $\mathrm{res}_x(A)dx$ is also of the form
  \[\Big(\frac{1}{(2\pi)^d}\int_{|\xi|=1}\mathrm{Tr}(\mathrm{Sub}(A)(x,\xi))dS(\xi)\Big)dx,\]
 which is independent of the choice of local coordinates of $M$ and local frames of $E$.
 For this reason,  if
 $A$ is a classical pseudodifferential operator of order $1-d$, then $\mathrm{res}(A)$ can be written as
  \begin{equation}\label{sub wodzicki}
\mathrm{res}(A)=\frac{1}{(2\pi)^d}\int_{T_1^{\ast}M}\mathrm{Tr}(\mathrm{Sub}(A)),
  \end{equation}
which is simply short for
\[\int_M\Big(\frac{1}{(2\pi)^d}\int_{|\xi|=1}\mathrm{Tr}(\mathrm{Sub}(A)(x,\xi))dS(\xi)\Big)dx.\]
Similarly, any integration over the unit cotangent bundle in the paper is always understood in this manner.

\begin{Proposition}\label{lemma 43} Let $A,B$ be classical pseudodifferential operators on sections of $E$  such that the sum of
the order of $A$ and $B$ is $1-d$. Then
\[\mathrm{res}(AB)=\frac{1}{(2\pi)^d}\int_{T_1^{\ast}M}\mathrm{Tr}(\mathrm{Sub}(A)\cdot\sigma_B^{(0)}+\sigma_A^{(0)}\cdot\mathrm{Sub}(B)).\]
\end{Proposition}

\begin{proof}
It is
straightforward to verify that (see also \cite[(1.4)]{Duistermaat Guillemin 1975})
\begin{equation}\label{sub product rule}
\mathrm{Sub}(AB)=\mathrm{Sub}(A)\cdot\sigma_B^{(0)}+\sigma_A^{(0)}\cdot\mathrm{Sub}(B)+\frac{1}{2\mathrm{i}}\{\sigma_A^{(0)},\sigma_B^{(0)}\},
\end{equation}
where \[\{\sigma_A^{(0)},\sigma_B^{(0)}\}=\frac{\partial\sigma_A^{(0)}}{\partial\xi_k}\cdot\frac{\partial\sigma_B^{(0)}}{\partial x^k}
-\frac{\partial\sigma_A^{(0)}}{\partial x^k}\cdot\frac{\partial\sigma_B^{(0)}}{\partial\xi_k}.\]
By considering the fact $\mathrm{Tr}(\{\sigma_A^{(0)},\sigma_B^{(0)}\}+\{\sigma_B^{(0)},\sigma_A^{(0)}\})=0$, one can easily get
\[\frac{\mathrm{Tr}(\mathrm{Sub}(AB)+\mathrm{Sub}(BA))}{2}=\mathrm{Tr}(\mathrm{Sub}(A)\cdot\sigma_B^{(0)}+\sigma_A^{(0)}\cdot\mathrm{Sub}(B)),\]
which proves the proposition via (\ref{sub wodzicki}) as $\mathrm{res}(AB)=\mathrm{res}(BA)$. We are done.
\end{proof}

According to Proposition \ref{prop 28} and Proposition \ref{lemma 43}, one has
\begin{equation}\label{47}
\mathscr{A}_1(A,P)=\frac{1}{(2\pi)^d}\int_{T_1^{\ast}M}\mathrm{Tr}\big(\mathrm{Sub}(A)
\cdot\sigma_{P^{\frac{1-d-m}{2}}}^{(0)}+\sigma_A^{(0)}\cdot\mathrm{Sub}(P^{\frac{1-d-m}{2}})
\big).
\end{equation}
For any real number $q$, it is not hard to verify that (see also \cite[(1.3)]{Duistermaat Guillemin 1975} and its proof therein)
\begin{equation}\label{499}
\mathrm{Sub}(P^{q})=q\cdot(\sigma_{P}^{(0)})^{q-1}\cdot\mathrm{Sub}(P).
\end{equation}
Thus by combining (\ref{47}) $\sim$ (\ref{499})
and by considering  $\sigma_{P^{q}}^{(0)}=(\sigma_{P}^{(0)})^{q}$,
we get

\begin{theorem}\label{theorem 47} Let $A$ be a classical pseudodifferential operator of order $m$ and let $P$ be a non-negative self-adjoint Laplacian
on sections of $E$. Then
\begin{equation}
\mathscr{A}_1(A,P)=\frac{1}{(2\pi)^d}\int_{T_1^{\ast}M}\mathrm{Tr}\big(\mathrm{Sub}(A)
+\frac{1-d-m}{2}\cdot\sigma_A^{(0)}\cdot\mathrm{Sub}(P)
\big).
\end{equation}
\end{theorem}

Theorem \ref{theorem 1.62} is an immediate consequence of Theorem \ref{theorem 47}
as it suffices to first note from (\ref{220}) that
\begin{equation}\label{A1}\mathscr{A}_1(A,D)=\frac{\mathscr{A}_1(A,D^2)+\mathscr{A}_1(AD,D^2)}{2},\end{equation}
then apply Theorem \ref{theorem 47} accordingly.

 Theorem \ref{theorem 1.6} is also a consequence of  Theorem \ref{theorem 47}, which can be seen as follows.

 \textsc{Step 1}:
Let $F$ be a smooth bundle endomorphism and let $D$ be a self-adjoint Dirac type operator on sections of $E$.
Obviously, it follows from Theorem \ref{theorem 47}  that $\mathscr{A}_1(F,D^2)=0$.
So according to (\ref{A1}), we get $\mathscr{A}_1(F,D)=\frac{\mathscr{A}_1(FD,D^2)}{2}$. By Proposition \ref{prop 28}, we also have
\[\mathscr{A}_1(FD,D^2)=\mathrm{res}(FD|D|^{-d})=\mathrm{res}(F|D|^{-d}D)=\mathrm{res}(DF|D|^{-d})=\mathscr{A}_1(DF,D^2).\]
Consequently,
\begin{equation}\mathscr{A}_1(F,D)=\frac{\mathscr{A}_1(FD,D^2)+\mathscr{A}_1(DF,D^2)}{4}.\end{equation}
Application of Theorem \ref{theorem 47} results in
\[\mathscr{A}_1(F,D)=\frac{1}{(2\pi)^d}\int_{T_1^{\ast}M}\mathrm{Tr}\Big(F\cdot \Big[\frac{\mathrm{Sub}(D)}{2}-
\frac{d}{2}\cdot\frac{\sigma_D^{(0)}\cdot\mathrm{Sub}(D^2)+\mathrm{Sub}(D^2)\cdot\sigma_D^{(0)}}{4}\Big]
\Big).\]

\textsc{Step 2}:
Recall $D=\gamma\widetilde{\nabla}+\psi$ is a self-adjoint Dirac type operator with potential $\psi$
associated with the Dirac bundle $(E,\gamma,\widetilde{\nabla})$. In local coordinates we  assume
$\widetilde{\nabla}_{\partial_j}=\partial_j+b_j,$
where $b_j$ are smooth matrix-valued functions. Thus $\sigma_D^{(0)}(x_0,\xi)=\mathrm{i}\gamma^j(x_0)\xi_j$,
$\sigma_D^{(1)}(x_0,\xi)=\gamma^j(x_0)b_j(x_0)+\psi(x_0).$
The compatible condition (\ref{compatible original}) between $\gamma$ and $\widetilde{\nabla}$ gives
$
\frac{\partial\gamma^k}{\partial x^j}=[\gamma^k,b_j]-\Gamma_{jn}^k\gamma^n.
$
We further assume the local coordinate system is a  Riemannian normal one centered at $x_0\in M$. Hence
\[
\big(\frac{\partial\gamma^k}{\partial x^j}\big)(x_0)=[\gamma^k(x_0),b_j(x_0)].
\]
So by the definition of the sub-principal symbol of $D$, we get
\begin{equation}\mathrm{Sub}(D)(x_0,\xi)=\gamma^j(x_0)b_j(x_0)+\psi(x_0)-\frac{1}{2}\big(\frac{\partial\gamma^k}{\partial x^k}\big)(x_0)=\big(\frac{\gamma^jb_j+b_j\gamma^j}{2}+\psi\big)(x_0).\end{equation}
Similarly, one can show that $\sigma_{D^2}^{(0)}(x_0,\xi)=|\xi|^2$, 
$
 \sigma_{D^2}^{(1)}(x_0,\xi)=\mathrm{i}(\gamma^k\psi+\psi\gamma^k-2b_k)(x_0)\xi_k.
$
So by the definition of the sub-principal symbol of $D^2$, we get
\[\mathrm{Sub}(D^2)(x_0,\xi)=\mathrm{i}(\gamma^k\psi+\psi\gamma^k-2b_k)(x_0)\xi_k.\]
Consequently,
 \begin{align}(\sigma_{D}^{(0)}\cdot\mathrm{Sub}(D^2))(x_0,\xi)&=-(\gamma^j\gamma^k\psi+\gamma^j\psi\gamma^k-2\gamma^jb_k)(x_0)\xi_j\xi_k,\\
 (\mathrm{Sub}(D^2)\cdot\sigma_{D}^{(0)})(x_0,\xi)&=-(\gamma^k\psi\gamma^j+\psi\gamma^k\gamma^j-2b_k\gamma^j)(x_0)\xi_k\xi_j.
 \end{align}

\textsc{Step 3}: By the above formulae we obtain
\begin{align*}
\int_{|\xi|=1}\mathrm{Tr}(F\cdot\mathrm{Sub}(D))(x_0,\xi)dS(\xi)&=\mathrm{Tr}\big(F\frac{\gamma^jb_j+b_j\gamma^j}{2}+F\psi\big)(x_0)\cdot\mathrm{Vol}(\mathbb{S}^{d-1}),\\
\int_{|\xi|=1}\mathrm{Tr}(F\cdot\sigma_D^{(0)}\cdot\mathrm{Sub}(D^2))(x_0,\xi)dS(\xi)&=\mathrm{Tr}\big(F\psi-\frac{F\widehat{\psi}}{d}+2\frac{F\gamma^jb_j}{d}\big)(x_0)\cdot\mathrm{Vol}(\mathbb{S}^{d-1}),\\
\int_{|\xi|=1}\mathrm{Tr}(F\cdot\mathrm{Sub}(D^2)\cdot\sigma_D^{(0)})(x_0,\xi)dS(\xi)&=\mathrm{Tr}\big(F\psi -\frac{F\widehat{\psi}}{d} +2\frac{Fb_j\gamma^j}{d}\big)(x_0)\cdot\mathrm{Vol}(\mathbb{S}^{d-1}).
\end{align*}
Therefore, by the last formula in Step 1,
\[\mathscr{A}_1(F,D)=\frac{1}{(2\pi)^d}\int_{M}\mathrm{Tr}\big(F\frac{\widehat{\psi}-(d-2)\psi}{4}\big)\cdot\mathrm{Vol}(\mathbb{S}^{d-1}),\]
which proves Theorem \ref{theorem 1.6} as
 $\mathrm{Vol}(\mathbb{S}^{d-1})=\frac{2\pi^{d/2}}{\Gamma(d/2)}$.

\section{Dirac type operators of vanishing second coefficient}\label{section 5}

Let
$D=\gamma\widetilde{\nabla}+\psi$ be a Dirac type operator of
potential $\psi$ associated with the Dirac bundle $(E,\gamma,\widetilde{\nabla})$. In this section we will no longer assume
the self-adjointness of $D$.
According to Remark \ref{remark self adjoint},
$\mathscr{H}_1(D,D^2)\in C^{\infty}(M;\mathrm{End}(E))$ is well-defined.

\begin{theorem}\label{theorem 51}  Let
 $D$ be a Dirac type operator.
 Then
${\mathscr H}_1(D,D^2)=0$ if and only if $D$ is a generalized Dirac operator.
\end{theorem}

By considering (\ref{L1D}) and (\ref{114}), we see that Theorem \ref{theorem 51} implies Theorem \ref{theorem 1.3} as a self-adjoint Dirac type operator
is a Dirac type operator. But we should remark that Theorem \ref{theorem 51} and Theorem \ref{theorem 1.3} actually are equivalent
as for any Dirac bundle one can always introduce a hermitian structure so that the associated generalized
Dirac operator is self-adjoint on the Hilbert space defined by this hermitian structure and the metric measure on $M$ (see e.g. \cite{Branson Gilkey}).

Now we explain how to prove Theorem \ref{theorem 51}.
Let $D$ be a Dirac type operator of potential $\psi$ associated with the Dirac bundle $(E,\gamma,\nabla)$.
According to Remark \ref{remark self adjoint}, $\mathscr{H}_1(D,D^2)=0$ if and only if $\widehat{\psi}=(d-2)\psi$.
Thus to prove Theorem \ref{theorem 51} it suffices to show that if $\widehat{\psi}=(d-2)\psi$, then there exists a connection $\widetilde{\nabla}$ on $E$ compatible with $\gamma$ such that $D=\gamma\widetilde{\nabla}$.
Note any connection on $E$ must be of the form $\nabla_L=\nabla+L$, where
\[L:C^{\infty}(M;TM)\rightarrow C^{\infty}(M;\mathrm{End}(E))\] is a $C^{\infty}(M)$-linear map.
 It is easy to check that $\nabla_L$ is compatible with $\gamma$ if and only if
$L(X)$ commutes with $\gamma(Y)$ for all $X,Y\in C^{\infty}(M;TM)$. Suppose we do have such a map $L$ such  that $D=\gamma\nabla_L$.
Letting
$\{e_k\}_{k=1}^d$ be a local orthonormal frame in $TM$, we have $D=\gamma\nabla+\gamma(e_k)L(e_k)$, which means $\psi=\gamma(e_k)L(e_k)$.
Consequently, for any fixed $i\in\{1,\ldots,d\}$,
\begin{align*}
\frac{\gamma(e_i)\psi+\psi\gamma(e_i)}{2}=L(e_k)\frac{\gamma(e_i)\gamma(e_k)+\gamma(e_k)\gamma(e_i)}{2}=-L(e_i).
\end{align*}
This means that globally $L$ must be uniquely of the following form
\begin{equation}\label{51}L(X)=-\frac{\gamma(X)\psi+\psi\gamma(X)}{2}\end{equation}
for any $X\in C^{\infty}(M;TM)$.
Thus to prove Theorem \ref{theorem 51}, with assuming $\widehat{\psi}=(d-2)\psi$ and  defining  $\widetilde{\nabla}=\nabla+L$ with $L$ given by (\ref{51}), we need only to show that 1)  $D=\gamma\widetilde{\nabla}$, and 2) $\widetilde{\nabla}$ is compatible with $\gamma$.

The first property can be verified in the following way. Let $\{e_k\}_{k=1}^d$ be a local orthonormal frame in $TM$, then
\[
\gamma(e_k)L(e_k)=-\gamma(e_k)\frac{\gamma(e_k)\psi+\psi\gamma(e_k)}{2}=\frac{d\psi-\widehat{\psi}}{2}=\psi,
\]
which gives $D=\gamma\nabla+\psi=\gamma\nabla+\gamma(e_k)L(e_k)=\gamma\widetilde{\nabla}$.

To prove the second property we need to show that $L(X)$ commutes with $\gamma(Y)$ for all $X,Y\in C^{\infty}(M;TM)$.
But considering $L$ and $\gamma$ are $C^{\infty}(M)$-linear maps, it suffices to prove that
$L_x(X_x)$ commutes with $\gamma_x(Y_x)$ for all  $X_x,Y_x\in T_xM$ at each $x\in M$. This is guaranteed by the next proposition and thus a proof of Theorem \ref{theorem 51} is achieved.

 \begin{Proposition}\label{prop 51} Let  $(W,\gamma)$ be a complex $\mathrm{Cl}(\mathbb{R}^d)$-module. Let $\psi\in\mathrm{End}(W)$ be such that
 $\widehat{\psi}=(d-2)\psi$, and let $L:\mathbb{R}^d\rightarrow\mathrm{End}(W)$ be the linear map defined by
 \[L(X)=-\frac{\gamma(X)\psi+\psi\gamma(X)}{2}\ \ \ (X\in\mathbb{R}^d).\]
 Then $L(X)$ commutes with $\gamma(Y)$ for all $X,Y\in\mathbb{R}^d$.
 \end{Proposition}

The rest of this section is mainly devoted to proving Proposition \ref{prop 51}.
In the first subsection we begin with reviewing a few basic concepts about the representations of irreducible complex Clifford modules (see e.g. \cite{Schroder}), then give a proof of Proposition \ref{prop 51}.
A few examples such as the massless Dirac operators (\cite{ChervovaDV,ChervovaV}) are studied in the second part.


\subsection{Clifford modules}



 Let $\mathrm{Cl}(\mathbb{R}^d)$ be the complex Clifford algebra generated by $(\mathbb{R}^d,\langle\cdot,\cdot\rangle)$
subject to the commutation relation
$
X\ast Y+Y\ast X=-2\langle X,Y\rangle,
$
where $X,Y\in \mathbb{R}^d$, $\langle\cdot,\cdot\rangle$ denotes the standard Euclidean metric on $\mathbb{R}^d$, and $\ast$ denotes the Clifford algebra operation. Any  complex $\mathrm{Cl}(\mathbb{R}^d)$-module $W$
with structure  $\gamma$ (unital algebra morphism from $\mathrm{Cl}(\mathbb{R}^d)$ to $\mathrm{End}(W)$) studied in this paper is always assumed to be of dimension in $\mathbb{N}$.
 A  complex $\mathrm{Cl}(\mathbb{R}^d)$-module $(W,\gamma)$
 is said to be \emph{irreducible}
if for any decomposition $W=W_1\oplus W_2$ into subspaces invariant under $\gamma$ one has
$W_1=W$ or $W_2=W$. It is known that a  complex $\mathrm{Cl}(\mathbb{R}^d)$-module is irreducible if and only if it is of complex dimension $2^{\lfloor\frac{d}{2}\rfloor}$.
 Two complex $\mathrm{Cl}(\mathbb{R}^d)$-modules $(W_1,\gamma)$ and $(W_2,\widetilde{\gamma})$
 are said to be \emph{equivalent} if there exists an invertible element $\kappa\in\mathrm{Hom}(W_1,W_2)$
such that $\widetilde{\gamma}=\kappa^{\ast}\gamma$, where $\kappa^{\ast}$ is defined by
sending $\varrho\in\mathrm{End}(W_1)$ to $\kappa\varrho\kappa^{-1}\in\mathrm{End}(W_2)$.
 Up to isomorphism there are exactly $\frac{3-(-1)^d}{2}$ inequivalent irreducible  complex $\mathrm{Cl}(\mathbb{R}^d)$-modules.
It is also known that any complex
$\mathrm{Cl}(\mathbb{R}^d)$-module  is a sum of irreducible ones.

Let $(W,\gamma)$ be a   complex $\mathrm{Cl}(\mathbb{R}^d)$-module and define
$\widehat{\psi}=\sum_{k=1}^d \gamma(e_k)\psi\gamma(e_k)$
for any $\psi\in\mathrm{End}(W)$,
where $\{e_k\}_{k=1}^d$ is an arbitrary orthonormal basis for $(\mathbb{R}^d,\langle\cdot,\cdot\rangle)$.
Let
$$\mathrm{End}_k(W):=\mathrm{Span}_{\mathbb{C}}\big\{\gamma(e_{i_1})\gamma(e_{i_2})\cdots\gamma(e_{i_k}): 1\leq i_1<i_2<\cdots<i_k\leq d\big\},$$
for $k>0$ and let $\mathrm{End}_0(W)$ be the complex span of the identity map.
A simple computation using the Clifford algebra relations shows that
$\mathrm{End}_k(W)$ is an eigenspace for the map $\psi \mapsto \widehat{\psi}$ with eigenvalue $(-1)^k(2k-d)$.
It is clear that Clifford multiplication by the volume element $\gamma(e_1) \gamma(e_2)\cdots\gamma(e_d)$
defines a linear map from $\mathrm{End}_k(W)$ to $\mathrm{End}_{d-k}(W)$.
If the module is irreducible the subspaces $\mathrm{End}_k(W)$ generate $\mathrm{End}(W)$ and therefore this
gives a decomposition into eigenspaces. Hence, in the irreducible case we have
\begin{gather*}
 \mathrm{End}(W) = \bigoplus_{k=0}^d \mathrm{End}_k(W) \quad \textrm{if}\; d\; \textrm{is even},\\
  \mathrm{End}(W) = \bigoplus_{k=0}^{\frac{d-1}{2}} \mathrm{End}_k(W) \quad  \textrm{if}\; d\; \textrm{is odd}.
\end{gather*}
In the latter case we have used $\mathrm{End}_{d-k}(W)=\mathrm{End}_k(W)$. This can be seen direclty because
Clifford multiplication by the volume element commutes with the Clifford  action. Therefore,
by Schur's lemma, it must be a multiple of the identity.
As an immediate consequence of this discussion we get the following Proposition.
\begin{Proposition}\label{prop 53}
If $(W,\gamma)$ is irreducible, then the $d-2$ eigenspace of the map  $\psi \mapsto \widehat{\psi}$ equals
$\mathrm{End}_{1}(W)=\mathrm{Span}_{\mathbb{C}}\gamma(\mathbb{R}^d)$.
If furthermore $d$ is  odd, then
$2-d$ is not an eigenvalue of the map $\psi \mapsto \widehat{\psi}$.
\end{Proposition}

Now  we can start to prove Proposition \ref{prop 51}.
 Let  $(W,\gamma)$ be a (finite-dimensional) complex $\mathrm{Cl}(\mathbb{R}^d)$-module
and let $(W_0,\gamma^{(0)})$ be an  irreducible complex $\mathrm{Cl}(\mathbb{R}^d)$-module. It is easy to see that up to isomorphism $(W_0,\gamma^{(0)})$ can represent all irreducible  complex $\mathrm{Cl}(\mathbb{R}^d)$-modules if $d$ is even,
 and so for $(W_0,\pm\gamma^{(0)})$ if $d$ is odd.
We then have two cases to consider.

Case 1: Suppose $d$ is even. In this case
 there exists a (finite-dimensional)
 complex vector space $V$ such that $(W,\gamma)$ is equivalent to $(W_0\otimes V,\widetilde{\gamma})$, where the Clifford
 structure $\widetilde{\gamma}$ on $W_0\otimes V$ is applied only to the first tensor factor $W_0$.
 To be precise, $\widetilde{\gamma}(Y)=\gamma^{(0)}(Y)\otimes\mathrm{Id}_V$ for any $Y\in\mathbb{R}^d$. Without loss of generality
 we can assume that $(W,\gamma)=(W_0\otimes V,\widetilde{\gamma})$ since the general cases can be dealt with simply by studying isomorphism between  Clifford modules. Then as an application of Proposition \ref{prop 53},
 we claim that the $d-2$ eigenspace of the map  $\psi \mapsto \widehat{\psi}$
 equals $\gamma^{(0)}(\mathbb{R}^d)\otimes\mathrm{End}(V)$. To prove this claim,
 it suffices to first fix a linear basis $\{K_i\}$ for $\mathrm{End}(V)$,
then  express  $\psi$ with $\widehat{\psi}=(d-2)\psi$ uniquely as
$\psi_i^{(0)}\otimes K_i$
where $\psi_i^{(0)}\in\mathrm{End}(W_0)$,  note \[\widehat{\psi_i^{(0)}}\otimes K_i=\widehat{\psi}=(d-2)\psi=(d-2)\psi_i\otimes K_i,\] and finally apply Proposition \ref{prop 53} appropriately.
This claim immediately implies that
 \[L(\mathbb{R}^d)=\{\mathrm{Id}_{W_0}\}\otimes\mathrm{End}(V),\]
 which obviously commutes with $\gamma(\mathbb{R}^d)=\gamma^{(0)}(\mathbb{R}^d)\otimes\{\mathrm{Id}_V\}$.

 Case 2: Suppose $d$ is odd.  In this case
 there exist (finite-dimensional)
 complex vector space $V_1,V_2$ such that $(W,\gamma)$ is equivalent to $((W_0\otimes V_1)\oplus(W_0\otimes V_2),\widetilde{\gamma})$, where the Clifford
 structure $\widetilde{\gamma}$ on $(W_0\otimes V_1)\oplus(W_0\otimes V_2)$ applied  only to $W_0$, is defined by
 \[\widetilde{\gamma}(Y)=\big(\gamma^{(0)}(Y)\otimes\mathrm{Id}_{V_1}\big)\bigoplus\big(-\gamma^{(0)}(Y)\otimes\mathrm{Id}_{V_2}\big)\]
 for any $Y\in\mathbb{R}^d$. Similar to the discussions in the previous case, we assume without loss of generality that
  $(W,\gamma)=((W_0\otimes V_1)\oplus(W_0\otimes V_2),\widetilde{\gamma})$.  Then as an application of Proposition \ref{prop 53},
 we claim that the $d-2$ eigenspace of the map  $\psi \mapsto \widehat{\psi}$
 equals $(\gamma^{(0)}(\mathbb{R}^d)\otimes\mathrm{End}(V_1))\oplus(\gamma^{(0)}(\mathbb{R}^d)\otimes\mathrm{End}(V_2))$.
 To prove this claim,
we first choose a linear basis $\{K_i^{\alpha\beta}\}$ for each $\mathrm{Hom}(V_{\alpha},V_{\beta})$,
then  express  $\psi$ with $\widehat{\psi}=(d-2)\psi$ uniquely as
\[\big(\psi_i^{(11)}\otimes K_i^{(11)}\big)\bigoplus\big(\psi_j^{(22)}\otimes K_j^{(22)}\big)+
\big(\psi_i^{(12)}\otimes K_i^{(12)}\big)\bigoplus\big(\psi_j^{(21)}\otimes K_j^{(21)}\big),\]
where $\psi_i^{(\alpha\beta)}\in\mathrm{End}(W_0)$. Here to be clear,  $W=(W_0\otimes V_1)\oplus(W_0\otimes V_2)$
is naturally identified with $(W_0\otimes V_2)\oplus(W_0\otimes V_1)$ by interchanging the positions,
 thus any element in $W$  mapped via the second summand
is indeed contained in $W$. Following this convention, it is straightforward to check that
\[\widehat{\psi}=\big(\widehat{\psi_i^{(11)}}\otimes K_i^{(11)}\big)\bigoplus\big(\widehat{\psi_j^{(22)}}\otimes K_j^{(22)}\big)+
\big(-\widehat{\psi_i^{(12)}}\otimes K_i^{(12)}\big)\bigoplus\big(-\widehat{\psi_j^{(21)}}\otimes K_j^{(21)}\big).\]
Thus the claim is an immediate consequence of Proposition \ref{prop 53}.
This claim implies that
 \[L(\mathbb{R}^d)=\big(\{\mathrm{Id}_{W_0}\}\otimes\mathrm{End}(V_1)\big)\bigoplus\big(\{\mathrm{Id}_{W_0}\}\otimes\mathrm{End}(V_2)\big),\]
 which commutes with $\widetilde{\gamma}(Y)$ for any $Y\in\mathbb{R}^d$.
This finishes the proof of Proposition \ref{prop 51}.


\subsection{Examples}

In this part we study a few examples. Recall  the diagonal values of the   smooth kernel  of $De^{-tD^2}$
admit a uniform small-time asymptotic expansion
\[
K(t,x,x,D,D^2)\sim\sum_{k=0}^{\infty}t^{\frac{k-d-1}{2}}{\mathscr H}_k(D,D^2)(x)\ \ \ (t\rightarrow0^{+}),
\]
where
\[
\mathscr{H}_1(D,D^2)=
\frac{1}{(4\pi)^{d/2}}\cdot \frac{\widehat{\psi}-(d-2)\psi}{2}.
\]
If $D$ is further self-adjoint, then the mollified local counting function of $D$
has a pointwise asymptotic expansion
\[
(\chi\ast N_x' )(\mu)\sim\sum_{k=0}^{\infty}\mu^{d-k-1}\mathrm{Tr}_{E_x}(\mathscr{L}_k(D)(x))\ \ \ (\mu\rightarrow\infty),
\]
where
$\mathscr{L}_1(D)=\frac{\mathscr{H}_1(D,D^2)}{\Gamma(\frac{d}{2})}$.

\begin{example}[Generalized Dirac operators]\label{example 33}
Let $D$ be a generalized Dirac operator  associated with a Dirac bundle. Obviously,  $\mathscr{H}_1(D,D^2)=0$. To compare, it was shown in \cite{Branson Gilkey} that
$\mathrm{Tr}(\mathscr{H}_1(D,D^2))=0$.
\end{example}

\begin{example}[Three-dimensional manifolds] \label{example 55}
Let $D$ be a Dirac type operator of potential $\psi$ associated with a Dirac bundle of rank two over a three-dimensional closed Riemannian manifold $M$.
  We claim $\mathscr{H}_1(D,D^2)=0$ if and only if $\mathrm{Tr}(\psi)=0$. To prove this claim, we  note that if $W$ is an irreducible complex
   $\mathrm{Cl}(\mathbb{R}^{3})$-module, then $W$ is of complex dimension $2$ and
   $ \mathrm{End}(W)= \mathrm{End}_{0}(W) \oplus  \mathrm{End}_{1}(W)$. Consequently,  $\mathrm{End}_{1}(W)$ is precisely the trace-free part of $\mathrm{End}(W)$.
\end{example}

\begin{example}[Parallelizable manifolds] \label{example 56}
In this example we construct generalized Dirac operators acting on $C^{\infty}(M;E)$, where  $M$
is a closed parallelizable Riemannian manifold of dimension $d$, $E=M\times \mathbb{C}^{r}$ with rank $r=2^{\lfloor\frac{d}{2}\rfloor}$.
The Riemannian metric on $M$ is denoted by $g$ and the Riemannian connection on $TM$ is denoted by $\nabla$.
Since $M$ is  parallelizable, there exist
smooth real vector fields $\{X_k\}_{k=1}^d$ such that $\{X_k(x)\}_{k=1}^d$ is an orthonormal basis for $(T_xM,g_x)$ at each $x\in M$.
Let $\{R_k\}_{k=1}^d$ be fixed complex matrices of size $r\times r$ with
 $R_jR_k+R_kR_j=-2\delta_{jk}$ for all $1\leq j,k\leq d$. Then it is straightforward to verify that $\gamma:TM\rightarrow\mathrm{End}(E)$ defined by
   \begin{equation}\gamma(X)=g(X,X_k)R_k \end{equation}
 is an irreducible $\mathrm{Cl}(TM)$-module structure.
Define a flat connection $\widehat{\nabla}$ on $E$  by
\begin{equation}\widehat{\nabla}_X\phi=\widehat{\nabla}_X\left(\begin{array}{c}
\phi_1  \\
\vdots\\
\phi_r
 \end{array}
 \right)=\left(\begin{array}{c}
X\phi_1  \\
\vdots\\
X\phi_r
 \end{array}
 \right)\end{equation}
 and define  $L:TM\rightarrow\mathrm{End}(E)$ by
 \begin{equation}L(X)=\frac{R_k\gamma(\nabla_XX_k)}{4}.\end{equation}
 Following Branson-Gilkey (\cite[Lemma 1.3]{Branson Gilkey}),  we claim  $\widetilde{\nabla}\triangleq\widehat{\nabla}+L$ is a compatible connection on $(E,\gamma)$. To this end we first let $X,\alpha\in C^{\infty}(M;TM)$, $\phi\in C^{\infty}(M;E)$. Obviously, $\alpha$ is of the form
 $\alpha=\alpha_kX_k$
 for some $\alpha_k\in C^{\infty}(M)$. Thus
 \begin{align*}
\widetilde{\nabla}_X(\gamma(\alpha)\phi)&=\widetilde{\nabla}_X(\alpha_kR_k\phi)=\widehat{\nabla}_X(\alpha_kR_k\phi)+L(X)\gamma(\alpha)\phi\\
&=(X\alpha_k)(R_k\phi)+\gamma(\alpha)\widehat{\nabla}_X\phi+L(X)\gamma(\alpha)\phi
 \end{align*}
and
 \begin{align*}
\gamma(\nabla_X\alpha)\phi+\gamma(\alpha)\widetilde{\nabla}_X\phi&=\gamma((X\alpha_k)X_k+
\alpha_k\nabla_XX_k)\phi+\gamma(\alpha)\widehat{\nabla}_X\phi+
\gamma(\alpha)L(X)\phi\\
&=(X\alpha_k)(R_k\phi)+\alpha_k\gamma(\nabla_XX_k)\phi+\gamma(\alpha)\widehat{\nabla}_X\phi+
\gamma(\alpha)L(X)\phi.
 \end{align*}
Hence to prove the compatible condition (\ref{compatible original})  currently for $\alpha\in C^{\infty}(M;TM)$ it suffices to show
$[L(X), \gamma(\alpha)]=g(\alpha, X_k)\gamma(\nabla_XX_k),$
which is obviously equivalent to
\begin{equation}\label{equivalent compa}[L(X_i), \gamma(X_j)]=g(X_j, X_k)\gamma(\nabla_{X_i}X_k)\ \ \ (1\leq i,j\leq d).\end{equation}
Letting $\Gamma_{ij}^k=-\Gamma_{ik}^j$ denote the Christoffel symbols, that is, $\nabla_{X_i}X_j=\Gamma_{ij}^kX_k$, we have
\begin{align*}
[L(X_i), \gamma(X_j)]&=\frac{1}{4}\big(R_k\gamma(\nabla_{X_i}X_k)R_j-R_jR_k\gamma(\nabla_{X_i}X_k)\big)\\
&=\frac{1}{4}\big(R_k\Gamma_{ik}^nR_nR_j-R_jR_k\Gamma_{ik}^nR_n\big)\\
&=\frac{1}{4}\big(-2R_k\Gamma_{ik}^n\delta_{nj}-R_k\Gamma_{ik}^nR_jR_n-R_jR_k\Gamma_{ik}^nR_n\big)\\
&=\frac{1}{2}(-R_k\Gamma^j_{ik}+\Gamma^n_{ij}R_n)
=\Gamma^n_{ij}R_n\\
&=\gamma(\nabla_{X_i}X_j)=g(X_j, X_k)\gamma(\nabla_{X_i}X_k),
\end{align*}
 which proves  (\ref{equivalent compa}). Generally, if $\alpha\in  C^{\infty}(M;\mathrm{Cl}(TM))$ is of the form
 $\alpha=\alpha^{(1)}\ast \alpha^{(2)}\ast \cdots \ast\alpha^{(n)}$ where $\alpha^{(k)}\in C^{\infty}(M;TM)$ $(1\leq k\leq n)$,
 then according to (\ref{clifford preverse}),
 \begin{align*}
\widetilde{\nabla}_X(\gamma(\alpha)\phi)&=\widetilde{\nabla}_X\big(\gamma(\alpha^{(1)})\Big[\gamma(\alpha^{(2)})\cdots\gamma(\alpha^{(n)})\phi\Big]\big)\\
&=\gamma(\nabla_X\alpha^{(1)})\Big[\gamma(\alpha^{(2)})\cdots\gamma(\alpha^{(n)})\phi\Big]+\gamma(\alpha^{(1)})\widetilde{\nabla}_X\Big[\gamma(\alpha^{(2)})\cdots\gamma(\alpha^{(n)})\phi\Big]\\
&=\cdots\\
&=\sum_{k=1}^n\gamma(\alpha^{(1)})\cdots\gamma(\alpha^{(k-1)})\gamma(\nabla_X\alpha^{(k)})\gamma(\alpha^{(k+1)})\cdots\gamma(\alpha^{(n)})\phi+
\gamma(\alpha)\widetilde{\nabla}_X\phi\\
&=\gamma(\nabla_X\alpha)\phi+\gamma(\alpha)\widetilde{\nabla}_X\phi,
 \end{align*}
 which suffices to claim (\ref{compatible original}) for arbitrary $\alpha\in  C^{\infty}(M;\mathrm{Cl}(TM))$ by linearity.
  Thus $(E,\gamma,\widetilde{\nabla})$ is a Dirac bundle and
  \begin{equation}\label{Dirac parallel}
  D=\gamma\widetilde{\nabla}=\gamma\widehat{\nabla}+\frac{R_iR_j\gamma(\nabla_{X_i}X_j)}{4}
  \end{equation} is the associated generalized Dirac operator.
By Theorem \ref{theorem 51}, we get $\mathscr{H}_1(D,D^2)=0$.
  It is well-known (\cite{Branson Gilkey}) that one can always furnish $(E,\gamma,\widetilde{\nabla})$ with a smooth hermitian fiber metric so that  $D$ is self-adjoint on $L^2(M;E)$ defined by this fiber metric and the metric measure on $M$.
  So in this context
   $\mathscr{L}_1(D)=0$.
Next we derive a local expression of $D$. To this end we first define a flat connection $\ddot{\nabla}$ on $\mathrm{End}(E)$ by
\begin{equation}\ddot{\nabla}_X\left(\begin{array}{ccc}
\phi_{11} & \cdots & \phi_{1r} \\
\vdots & \ddots & \vdots \\
\phi_{r1} & \cdots & \phi_{rr}
 \end{array}
 \right)=
\left(\begin{array}{ccc}
X\phi_{11} & \cdots & X\phi_{1r} \\
\vdots & \ddots & \vdots \\
X\phi_{r1} & \cdots & X\phi_{rr}
 \end{array}
 \right)\end{equation}
then introduce
\begin{equation}
\Psi(X,Y,Z,W)=\gamma(X)\gamma(Y)\big(\gamma(\nabla_ZW)-\ddot{\nabla}_Z(\gamma(W))\big),
\end{equation}
where  $X,Y,Z,W\in C^{\infty}(M;TM)$.  It is easy to verify that $\Psi$ is $C^{\infty}(M)$-linear in all of the four variables. This implies
 $\mathrm{Tr}_{2,4}(\mathrm{Tr}_{1,3}(\Psi))$
  is a smooth endomorphism of $E$, where $\mathrm{Tr}_{1,3}(\Psi)$ denotes the trace of $\Psi$
 with respect to the first and third variables, and $\mathrm{Tr}_{2,4}(\Psi)$ is understood in a similar way. Thus globally we have
\[
\mathrm{Tr}_{2,4}(\mathrm{Tr}_{1,3}(\Psi))=\Psi(X_i,X_j,X_i,X_j)=R_iR_j\gamma(\nabla_{X_i}X_j),
\]
which means the generalized Dirac operator $D$ can also be written as
\begin{equation}D=\gamma\widehat{\nabla}+\frac{\mathrm{Tr}_{2,4}(\mathrm{Tr}_{1,3}(\Psi))}{4}.\end{equation}
Hence in a local coordinate system $(x^1,\ldots,x^d)$,
$D$ is of the form
\begin{equation}
D=\gamma(dx^k)\frac{\partial}{\partial x^k}+\frac{1}{4} \gamma(dx^i)\gamma(\frac{\partial}{\partial x^j})\big(\gamma(\nabla_{\partial_i}dx^j)-\frac{\partial(\gamma(dx^j))}{\partial x^i}\big).
\end{equation}
 As an example, we assume that $M$ is a three-dimensional  oriented closed Riemannian manifold. By Steenrod's theorem $M$ is parallelizable
  and we can set \[R_1=\left(\begin{array}{cc}
0 & -\mathrm{i} \\
 -\mathrm{i} & 0
 \end{array}
 \right),\ \ \ R_2=\left(\begin{array}{cc}
0 & -1 \\
 1 & 0
 \end{array}
 \right),\ \ \ R_3=\left(\begin{array}{cc}
-\mathrm{i} & 0 \\
0 & \mathrm{i}
 \end{array}
 \right).\]
 Then it is straightforward to verify that $D$ is nothing but the massless Dirac operator $W$ locally defined by (A.3)
 in \cite{ChervovaDV}. Note
 $\mathrm{Tr}_{E}(\mathscr{L}_1(D))=0$
 was established in \cite{ChervovaDV}. To compare, we have shown that $\mathscr{L}_1(D)=0$.
 Moreover, it was proved in \cite{ChervovaDV} that the sub-principal symbol of $D$
 is proportional to the identity matrix at each point of the manifold $M$, where
the corresponding bundle trivialization $\{s_1,s_2\}$
 of  $E=M\times \mathbb{C}^2$ is chosen to be
 \[s_1=\left(\begin{array}{c}
1  \\
0
 \end{array}
 \right),\ \ \ s_2=\left(\begin{array}{c}
0  \\
 1
 \end{array}
 \right).\]
 In the following we will give
 a short proof of this fact. According to Example \ref{Example Dirac} and our previous discussions, it is easy to see that $\omega=L$, and consequently,
 \begin{align*}
 \mathrm{Sub}(D)&=\frac{1}{2}(\gamma(X_j)L(X_j)+L(X_j)\gamma(X_j))\\
 &=\frac{1}{8}(R_jR_k\Gamma_{jk}^nR_n+R_k\Gamma_{jk}^nR_nR_j),
 \end{align*}
 where we recall the Christoffel symbols $\Gamma_{jk}^n$ are given by $\nabla_{X_j}X_k=\Gamma_{jk}^nX_n$ with $\Gamma_{jk}^n=-\Gamma_{jn}^k$.
 Note $\mathrm{Sub}(D)$ is a sum of $27=3\times 3\times 3$ items. Next we no longer use Einstein's sum convention  and decompose $\mathrm{Sub}(D)$ into four parts:
 \begin{align*}
 \mathrm{Sub}(D)=\sum_{k=n}\star+\sum_{k\neq n, j=n}\star+\sum_{k\neq n, j\neq n, k=j}\star+\sum_{k\neq n, j\neq n, k\neq j}\star\triangleq\mathrm{I}+\mathrm{II}+\mathrm{III}+\mathrm{IV},
  \end{align*}
where $\star$ is short for $\frac{1}{8}(R_jR_k\Gamma_{jk}^nR_n+R_k\Gamma_{jk}^nR_nR_j)$.
One can easily  deduce from the properties $\Gamma_{jk}^n=-\Gamma_{jn}^k$, $R_jR_k+R_kR_j=-2\delta_{jk}$ that
$\mathrm{I}=\mathrm{II}=\mathrm{III}=0$ and
\begin{align*}
\mathrm{IV}&=\sum_{(j,k,n)\in \pi_3}\frac{1}{8}(R_jR_k\Gamma_{jk}^nR_n+R_k\Gamma_{jk}^nR_nR_j)\\
&=\frac{1}{4}\sum_{(j,k,n)\in \pi_3}\Gamma_{jk}^nR_jR_kR_n\\
&=\frac{1}{2}\sum_{(j,k,n)\in \pi_3\ \mathrm{is\ even}}\Gamma_{jk}^nR_jR_kR_n\\
&=-\frac{\Gamma_{12}^3+\Gamma_{23}^1+\Gamma_{31}^2}{2},
\end{align*}
where $\pi_3$ denotes the symmetric group of degree 3, and the last equality follows from $R_1R_2R_3=R_2R_3R_1=R_3R_1R_2=-1.$
To conclude  we get
\begin{equation}
\mathrm{Sub}(D)=-\frac{\Gamma_{12}^3+\Gamma_{23}^1+\Gamma_{31}^2}{2},
\end{equation}
which agrees with \cite[Lemma 6.1]{ChervovaDV}. Furthermore, it is characterized in \cite{ChervovaDV} that
given a self-adjoint Dirac type operator $D$ on  sections of $M\times\mathbb{C}^2$, it is a massless one if and only if
$\mathrm{Tr}(\mathscr{L}_1(D))=0$ and $\mathrm{Sub}(D)$ is pointwise proportional to the identity matrix, where
the corresponding bundle trivialization $\{s_1,s_2\}$
 of  $M\times \mathbb{C}^2$ is chosen as before. We can also give a proof of this fact. Based on the previous discussions,
 it suffices to prove the necessary part. So let $D$ be a self-adjoint Dirac type operator such that $\mathrm{Tr}(\mathscr{L}_1(D))=0$ and $\mathrm{Sub}(D)$ is pointwise proportional to the identity matrix. Since $M$ is a parallelizable three-dimensional manifold,
 it is easy to see that there exists a unique massless Dirac operator $\gamma\widetilde{\nabla}$ sharing the same principal symbol with $D$ and
 a unique bundle endomorphism $\psi$ such that $D=\gamma\widetilde{\nabla}+\psi$. According to Example \ref{example 55}, $\mathrm{Tr}(\psi)=0$. Note also
 \[\mathrm{Sub}(D)=\mathrm{Sub}(\gamma\widetilde{\nabla})+\psi,\]
 which implies that $\psi$ is pointwise proportional to the identity matrix. Thus we must have $\psi=0$, and consequently, $D=\gamma\widetilde{\nabla}$ is indeed a massless Dirac operator.
\end{example}

\thanks{\emph{Acknowledgments.}
 Part of the work for this paper was carried out at the programme ``Modern Theory of Wave Equations" at the Schr\"odinger institute
and both authors are grateful to the ESI for support and the hospitality during their stay. The second author would also like to
thank the \emph{Hausdorff Center of Mathematics} in Bonn for the support for his stay during the trimester program
``Non-commutative Geometry and its Applications''. Both authors thank Dmitri Vassiliev for interesting discussions.}

\end{document}